\documentclass[%a4paper, 
11pt,  reqno]{amsart}
\pdfoutput=1

\usepackage[margin=1.2in,marginparwidth=1.5cm, marginparsep=0.5cm]{geometry}

\usepackage{booktabs} % nice tables
\usepackage{microtype}
\usepackage{amssymb}
\usepackage{mathrsfs}

\usepackage{color}
\usepackage[implicit=true]{hyperref}

\usepackage{xsavebox}

\usepackage{cases}%%%

\allowdisplaybreaks[2]

\definecolor{gr}{rgb}   {0.,   0.69,   0.23 }
\definecolor{bl}{rgb}   {0.,   0.5,   1. }
\definecolor{mg}{rgb}   {0.85,  0.,    0.85}
\definecolor{yl}{rgb}   {0.8,  0.7,   0.}
\definecolor{or}{rgb}  {0.7,0.2,0.2}

\newtheorem{theorem}{Theorem} [section]

\newtheorem{lemma}[theorem]{Lemma}
\newtheorem{proposition}[theorem]{Proposition}
\newtheorem{remark}[theorem]{Remark}

\newtheorem{definition}[theorem]{Definition}

\DeclareMathOperator*{\intt}{\int}

\DeclareMathOperator*{\supp}{supp}
\DeclareMathOperator{\med}{med}

\DeclareMathOperator{\Id}{Id}

\newcommand{\noi}{\noindent}
\newcommand{\Z}{\mathbb{Z}}
\newcommand{\R}{\mathbb{R}}

\newcommand{\T}{\mathbb{T}}

\let\Re=\undefined\DeclareMathOperator*{\Re}{Re}
\let\Im=\undefined\DeclareMathOperator*{\Im}{Im}

\let\P= \undefined
\newcommand{\P}{\mathbf{P}}

\newcommand{\Q}{\mathbf{Q}}

\newcommand{\Sb}{\pmb{\mathcal{S}}}

\newcommand{\E}{\mathbb{E}}

\newcommand{\F}{\mathcal{F}}

\newcommand{\Nf}{\mathfrak{N}}

\newcommand{\al}{\alpha}
\newcommand{\be}{\beta}
\newcommand{\dl}{\delta}

\newcommand{\Dl}{\Delta}
\newcommand{\eps}{\varepsilon}

\newcommand{\ld}{\lambda}

\newcommand{\s}{\sigma}
\newcommand{\Si}{\Sigma}
\newcommand{\ft}{\widehat}

\newcommand{\wt}{\widetilde}
\newcommand{\cj}{\overline}

\newcommand{\dt}{\partial_t}

\newcommand{\ta}{\theta}

\renewcommand{\l}{\ell}
\renewcommand{\o}{\omega}
\renewcommand{\O}{\Omega}

\newcommand{\les}{\lesssim}
\newcommand{\ges}{\gtrsim}

\newcommand{\jb}[1]
{\langle #1 \rangle}

\newcommand{\ind}{\mathbf 1}

\newcommand{\too}{\longrightarrow}

\newcommand{\N}{\mathbb{N}}
\newcommand{\NN}{\mathcal{N}}

\newtheorem*{ackno}{Acknowledgements}

\newcommand{\RR}{\mathcal{R}}

\numberwithin{equation}{section}
\numberwithin{theorem}{section}

\newcommand{\Hb}{\mathbf{H}}

\newcommand{\Hf}{\mathfrak{H}}

\newcommand{\TT}{\mathcal{T}}

\DeclareMathOperator{\Law}{Law}

\makeatletter
\@namedef{subjclassname@2020}{%
  \textup{2020} Mathematics Subject Classification}
\makeatother

\begin{document}
\baselineskip = 14pt

\title[Revisiting Bourgain's argument]
{Revisiting Bourgain's probabilistic construction of solutions to 
the 2-$d$ cubic NLS}

\author[T.~Oh and Y.~Wang]
{Tadahiro Oh and Yuzhao Wang}

%\address{
%Tadahiro Oh, School of Mathematics\\
% School of Mathematics and Statistics, Beijing Institute of Technology,
%Beijing 100081, China,
%and
%The University of Edinburgh\\
%and The Maxwell Institute for the Mathematical Sciences\\
%James Clerk Maxwell Building\\
%The King's Buildings\\
%Peter Guthrie Tait Road\\
%Edinburgh\\ 
%EH9 3FD\\
% United Kingdom
%}

\address{
Tadahiro Oh, School of Mathematics\\
The University of Edinburgh\\
and The Maxwell Institute for the Mathematical Sciences\\
James Clerk Maxwell Building\\
The King's Buildings\\
Peter Guthrie Tait Road\\
Edinburgh\\ 
EH9 3FD\\
 United Kingdom, 
 and 
 School of Mathematics and Statistics, Beijing Institute of Technology,
Beijing 100081, China
}

\email{hiro.oh@ed.ac.uk}

\address{
Yuzhao Wang\\
School of Mathematics\\
Watson Building\\
University of Birmingham\\
Edgbaston\\
Birmingham\\
B15 2TT\\ United Kingdom}

\email{y.wang.14@bham.ac.uk}



\dedicatory{Dedicated to Professor Yoshio Tsutsumi   on the
occasion of his seventieth birthday}

\begin{abstract}
In a seminal paper (1996), Bourgain proved invariance of the Gibbs measure
for the defocusing cubic nonlinear Schr\"odinger equation on the two-dimensional torus
by 
constructing
local-in-time solutions in a probabilistic manner.
In this note, we revisit and streamline his argument, using the random
tensor estimate developed
by Deng, Nahmod, and Yue (2022).

\end{abstract}

%\date{\today}
%%
%
\maketitle

\tableofcontents

\newpage

\section{Introduction}
\label{SEC:1}

\subsection{Overview}

In seminal works \cite{BO94, BO96}, Bourgain
initiated the probabilistic study of 
dispersive partial differential equations (PDEs)
in the context of invariant Gibbs measures.
In particular, in \cite{BO96}, 
he established invariance of the Gibbs measures
associated with  the following defocusing cubic 
nonlinear Schr\"odinger equation (NLS) on the two-dimensional torus
 $\T^2 = (\R / (2\pi \Z))^2$:
\begin{align}
i \dt u +  \Delta u =   |u|^{2}u, 
%\begin{cases}
%i \dt u +  \Delta u =   |u|^{2}u \\
%u|_{t = 0} = u_0,
%\end{cases}
\qquad (x, t) \in \T^2 \times \R.
\label{NLS1}
\end{align}

\noi
The key novelty of this breakthrough work  %\cite{BO96}
is the  probabilistic construction
of local-in-time solutions
to~\eqref{NLS1} with Gibbsian random initial data of low regularity.
Bourgain's work has had a huge impact in the field, 
leading to the study of probabilistic well-posedness
of dispersive PDEs, allowing us to go beyond
the limit of deterministic analysis;
see, for example, 
\cite{BT1, BT2, CO, BT3, LM, BOP2, Poc}.
See also  \cite{BOP4, Tzv1}
for (slightly outdated) surveys on the subject.
In particular, 
over the recent years, 
we have witnessed a significant rapid progress
in the study 
lying 
at the intersection of nonlinear dispersive PDEs
and  probability theory.
See 
\cite{OTh2, GKO, GKO2, Bring0, DNY2, DNY3, OOT1, Bring1,DNY4,  OOT2, BDNY}
for 
 recent well-posedness studies
on singular stochastic 
 dispersive PDEs, broadly interpreted, with 
rough random initial data and\,/\,or 
stochastic forcing, 
where various tools and ideas have been introduced.
Our main goal in this note
is to revisit
and streamline Bourgain's probabilistic construction 
of solutions in \cite{BO96} %to~\eqref{NLS1}
by using 
tools and ideas 
from this recent development, 
in particular 
the random
tensor estimate (Lemma~\ref{LEM:RT}) developed
by Deng, Nahmod, and Yue~\cite{DNY3}.

\subsection{Gibbs measure}
Gibbs measures in statistical physics
describe  thermal equilibria
and have been studied extensively.
By drawing an analogy to finite-dimensional Hamiltonian systems, 
a  Gibbs measure, formally given by 
\begin{equation} 
d \rho = Z^{-1}e^{- H(u)}  d u, 
\label{Gibbs1}
\end{equation}

\noi
%for NLS and gKdV
is expected to be  invariant  for a Hamiltonian PDE,  
where  $H(u)$ denotes a Hamiltonian of the equation.
\noi
The study of invariant Gibbs measures
for Hamiltonian PDEs, 
initiated by 
Lebowitz, Rose, and Speer \cite{LRS}
and Bourgain
\cite{BO94, BO96}, 
consists of two disjoint problems:

\smallskip

\begin{itemize}
\item[(i)]
Construction of a Gibbs measure 
as a probability measure on functions\,/\,distributions.

\vspace{1mm}

\item[(ii)]
Construction of dynamics on the support of a Gibbs measure.

\end{itemize}

\smallskip

In this subsection, we briefly discuss  
some aspects of
  the Gibbs measure construction
for 
 the defocusing cubic NLS \eqref{NLS1} on $\T^2$
 with the Hamiltonian:
\begin{align*}
H(u) = \frac{1}{2} \int_{\T^2} |\nabla u |^2 dx + \frac{1}{4}\int_{\T^2} |u|^4 dx. 
\end{align*}

\noi
In this case, the Gibbs measure $\rho$
corresponds to (the complex-valued version of)
the so-called $\Phi^4_2$-measure, 
whose construction goes back to 
 Nelson \cite{Nelson}
(see also \cite{DT, OTh1}), 
and is 
given by\footnote{Namely, 
instead of $\rho$ in \eqref{Gibbs1}, 
we consider a Gibbs measure of the form:  
\begin{equation*} 
d \rho = Z^{-1} e^{- H(u) - \frac 12 M(u)}  d u, 
\end{equation*}

\noi
where  $M(u) = \int_{\T^2}|u|^2 dx$ is the conserved mass.}
\begin{align}
d\rho = Z^{-1} e^{- \frac 14 \int_{\T^2} |u|^4 dx} d\mu
\label{Gibbs2}
\end{align} 
 
\noi
(modulo a renormalization on the potential energy).
Here, $\mu$ is the massive Gaussian free field on $\T^2$, 
formally given by 
\begin{align}\label{gauss0}
d \mu 
= Z^{-1} e^{-\frac 12 \| u\|_{H^{1} }^2    } du.
\end{align}

\noi
More precisely, 
$\mu$ is given as the induced probability measure under the map:\footnote{We endow $\T^2$  
 with the normalized Lebesgue measure $dx_{\T^2}=(2\pi)^{-2}dx$. With a slight abuse of notation, we still use $dx$ to denote the normalized Lebesgue measure on $\T^2$.}
\begin{align}
\o\in \O \longmapsto u(\o) = \sum_{n \in \Z^2 } \frac{ g_n(\o)}{\jb{n}} e_n, 
\label{map}
\end{align}

%\noi
%where 
%
%and $\ft u(n)$  denotes the Fourier coefficient of $u$. 

\noi
where 
$\jb{\,\cdot\,} = (1+|\cdot|^2)^\frac{1}{2}$, 
$e_n(x) = e^{i n\cdot x}$, 
and $\{ g_n \}_{n \in \Z^2}$
is a sequence of mutually independent standard complex-valued
Gaussian random variables on a probability space 
$(\O,\F,\mathbb P)$.
It is easy to see that $u$ in~\eqref{map} 
belongs almost surely to $H^s(\T^2) \setminus L^2(\T^2)$
for any $s < 0$, 
and thus 
 a renormalization 
 is needed
 on the interaction potential in \eqref{Gibbs2}.

 Given $N \in \N$,  let $\P_N$ be  the frequency projector
 onto the frequencies $\{|n|\le N\}$ given by 
\begin{align*}
\P_N f = 
\sum_{ |n| \leq N}  \ft f (n)  e_n.
\end{align*}

\noi
For each fixed $x \in \T^2$, 
$ \P_N u(x)$ is
a mean-zero complex-valued Gaussian random variable with variance:
\begin{align}
\s_N = \E\big[|\P_N u(x)|^2\big] = \sum _{\substack{n \in \Z^2\\|n| \le N}} \frac1{\jb{n}^2}
\sim  \log N \too \infty, 
\label{sigma1}
\end{align}

\noi
as $N\to\infty$. 
In the current complex-valued setting, 
we define the Wick renormalized  power
$:\! |\P_N u|^4\!:$
by setting
\begin{align}
\begin{split}
:\! |\P_N u(x)|^4 \!: 
%\, \stackrel{\text{def}}{=} 
& = 2 \s_N^2 L_2(\s_N^{-1}|\P_N u(x)|^2)\\
& = |\P_N u(x)|^4 - 4 \s_N |\P_N u(x)|^2 + 2\s_N^2
\end{split}
\label{Wick1}
\end{align}

\noi
for $x \in \T^2$, 
where $L_2(z) = \frac 12 z^2 - 2z + 1$ is the Laguerre polynomial of degree $2$.
We point out that, by writing 
\begin{align}
|\P_N u(x) |^4 = (\Re \P_N u(x))^4 
+ 2 (\Re \P_N u(x))^2  (\Im \P_N u(x))^2
 + (\Im \P_N u(x))^4, 
\label{Wick1a} 
\end{align}

\noi
the renormalization in \eqref{Wick1}
is equivalent to separately applying the usual real-valued Wick renormalization via the Hermite polynomials to 
the powers of 
$\Re \P_N u(x)$ and $\Im\P_N u(x)$ in~\eqref{Wick1a};
see \cite[Lemma~2.1]{OTh1}.
See also \cite[p.\,2966]{BO}
and \cite[p.\,107]{OS}.
With this renormalization, we define the truncated Gibbs measure $\rho_N$ by 
\begin{align}
d\rho_N = Z_N^{-1} \exp\bigg(- \frac 14 \int_{\T^2} :\! |\P_N u|^4\!:  dx\bigg) d\mu.
\label{Gibbs4}
\end{align}

\noi
Then, by 
the Wiener chaos estimate (Lemma \ref{LEM:hyp})
and Nelson's estimate, we can show that, as $N \to \infty$, 
the truncated Gibbs measure $\rho_N$ converges in total variation\footnote{In fact, 
the  density of the truncated Gibbs measure $\rho_N$ converges in $L^p(\mu)$ for any finite $p \ge 1$.} to 
the (unique) limiting Gibbs measure $\rho$, 
which we write as 
\begin{align}
d\rho = Z^{-1} \exp\bigg(- \frac 14 \int_{\T^2} :\! | u|^4\!:  dx\bigg) d\mu.
\label{Gibbs5}
\end{align}

\noi
See \cite[Section 2]{OTh1}
for details.

\begin{remark}\label{REM:1} \rm
(i) Note that the limiting Gibbs measure $\rho$ in \eqref{Gibbs5}
is equivalent to the Gaussian free field $\mu$
in \eqref{gauss0}.  Namely, $\rho$ and $\mu$
are mutually absolutely continuous
with respect to each other.
Moreover, given any $N \in \N$, 
$\rho_N$ and $\rho$ are equivalent.

\smallskip

\noi
(ii)
The discussion above is restricted to the defocusing case.
In fact, in the focusing case, 
it is known that the truncated Gibbs measure on $\T^2$
(with a Wick-ordered $L^2$-cutoff)
does not converge   to any  limit in total variation, even up to a subsequence.
See \cite{BS, OST2, GOTT}.
We also refer
readers to 
\cite{OST1, OOT2}
%and the references therein
for an overview of the Gibbs measure construction
for NLS, the nonlinear wave equation, and the generalized KdV equation.
%(for both the defocusing and focusing cases).

\end{remark}

\subsection{Renormalized NLS} % and Bourgain's trick}

We now turn to the well-posedness problem
and discuss the construction of invariant Gibbs dynamics for NLS \eqref{NLS1}.
%the Gibbs measure $\rho$ in \eqref{Gibbs5}.
Given $N \in \N$, it follows from~\eqref{Wick1} 
that the dynamical problem
associated with the truncated Gibbs measure~$\rho_N$ in \eqref{Gibbs4}
is given by the following truncated Wick-ordered NLS on $\T^2$:
\begin{align}
i \dt u_N +  \Delta u _N=  \P_N\big\{ ( |\P_N u_N|^{2}  - 2\s_N)  \P_N u_N\big\}
%\begin{cases}
%i \dt u_N +  \Delta u _N=  \P_N\big\{ ( |\P_N u_N|^{2}  - 2\s_N)  \P_N u_N\big\} \\
%u_N|_{t = 0} = u_{0, N} \  \ \text{with } \ \Law(u_{0, N}) =  \rho_N, 
%\end{cases}
\label{NLS1a}
\end{align}

\noi
with $\Law(u_N(0)) = \rho_N$, 
where $\s_N$ is as in \eqref{sigma1}
and $\Law(X)$ denotes the law of a random variable $X$.
Then, it is easy to see that the truncated NLS \eqref{NLS1a} is $\rho_N$-almost surely globally well-posed
and that $\rho_N$ is invariant under
the dynamics of \eqref{NLS1a};
see
\cite[Lemma 5.1]{OTh1}.
In view of Remark~\ref{REM:1}\,(i), 
we see that 
\eqref{NLS1a} is also  globally well-posed
almost surely
with respect to $\rho$.
We  recall 
the main statement in \cite{BO96} regarding invariance of 
the limiting Gibbs measure~$\rho$ in \eqref{Gibbs5} under the dynamics
of the defocusing cubic NLS \eqref{NLS1a} on $\T^2$.

\begin{theorem}[Bourgain \cite{BO96}]\label{THM:0}
Given  $N \in \N$, 
let $\Phi_N$ be the solution map for the truncated NLS \eqref{NLS1a}, 
sending initial data $u_N(0)$ to the solution $u_N$. 
Then, there exists a set\footnote{We fix a sufficiently small constant $\eps > 0$ in the rest of this note.}
 $\Si \subset H^{-\eps}(\T^2)$
with $\rho(\Si) = 1$ such that 
for each $u_0 \in \Si$, 
as $N \to \infty$, 
the solution $\Phi_N(u_0)$
to~\eqref{NLS1a}
converges 
to a non-trivial limit
$u$ with $u|_{t=0} = u_0$
in 
$C(\R;  H^{-\eps}(\T^2))$, 
endowed with the compact-open
topology in time.
Furthermore, we have $\Law(u(t)) = \rho$
for any $t \in \R$.

\end{theorem}

Formally speaking, 
Theorem \ref{THM:0}
states that the Gibbs measure $\rho$ in \eqref{Gibbs5}
is invariant under the following (formal) Wick-ordered NLS:
\begin{align*}
i \dt u +  \Delta u =  ( | u|^{2}  - 2\cdot \infty)   u, 
\end{align*}

\noi
where $\infty = \lim_{N\to \infty}\s_N$.

The proof of Theorem \ref{THM:0} in \cite{BO96} is based
on studying the following gauge-equivalent formulation.
Given $N \in \N$, 
let  $u_N$ be a solution to the truncated NLS \eqref{NLS1a}
and
define $\al_N$ by setting 
\begin{align}
\al_N = \|\P_N u_N(0)\|_{L^2}^2 - \s_N.
\label{C1}
\end{align}

\noi
Then, by setting
\begin{align}
v_N(t)  = \P_N e^{2 it \al_N} u_N(t) + \P_N^\perp u_N(t),
\label{C2} 
\end{align}

\noi
where $\P_N^\perp = \Id - \P_N$, 
we see  that 
$v_N$ satisfies 
the following truncated NLS:
\begin{align}
\begin{cases}
i \dt v_N +  \Delta v _N=  \P_N\big\{ ( |\P_N v_N|^{2}  - 2\int_{\T^2} |\P_N v_N|^2 dx )  \P_N v_N\big\}\\
v_N|_{t = 0} = u_N(0).
\end{cases}
%
%\begin{cases}
%i \dt v_N +  \Delta v _N=  \P_N\big\{ ( |\P_N v_N|^{2}  - 2\int |\P_N v_N|^2 dx )  \P_N v_N\big\} \\
%v_N|_{t = 0} = u_{0, N} \  \ \text{with } \ \Law(u_{0, N}) =  \rho_N, 
%\end{cases}
\label{NLS1c}
\end{align}

\noi
Note from 
\eqref{C2} that  $\rho_N$
is invariant under \eqref{NLS1c}.
By formally taking $N \to \infty$, 
we obtain the following renormalized NLS:\footnote{The equation
\eqref{NLS2}
is often referred to as the Wick-ordered NLS, 
since
its frequency-truncated version~\eqref{NLS1c} is equivalent to the (truncated) Wick-ordered
NLS \eqref{NLS1a}.
See, for example, \cite{Christ, GH, GO, FOW, OW2, OW3}.}
\begin{align}
\label{NLS2}
\textstyle
i \partial_t v + \Delta v = 
\Nf(v)
\stackrel{\text{def}}{=} \big( |v|^2 - 2 \int_{\T^2}|v|^2 dx \big) v.
%\begin{cases}
%i \partial_t v + \Delta v = \big( |v|^2 - 2 \int|v|^2 dx \big) v\\
%u|_{t=0} = u_0^\o 
%\end{cases}
\end{align}

\begin{theorem}\label{THM:1}

\textup{(i)}
The renormalized NLS \eqref{NLS2}
is $\mu$-almost surely locally well-posed, 
where $\mu$ is the Gaussian free field defined
in \eqref{gauss0}.
More precisely,  there exists a set
 $\Si \subset H^{-\eps}(\T^2)$
with $\mu(\Si) = 1$ such that 
for each $u_0 \in \Si$, 
there exists a unique local-in-time solution $v$ to \eqref{NLS2}
with $v|_{t = 0} = u_0$.
Furthermore, 
for each $u_0 \in \Si$, 
by possibly shrinking the local existence time $T = T(u_0) > 0$, 
the solution $v_N$ to \eqref{NLS1c}
with $v_N|_{t = 0} = u_0$ 
converges to $v$
in $C([-T, T];  H^{-\eps}(\T^2))$
as $N \to \infty$.

\smallskip

\noi
\textup{(ii)}
The renormalized NLS \eqref{NLS2}
is $\rho$-almost surely globally  well-posed
and the Gibbs measure~$\rho$ is invariant under the resulting dynamics.
More precisely,  there exists a set
 $\Si \subset H^{-\eps}(\T^2)$
with $\rho(\Si) = 1$ such that 
for each $u_0 \in \Si$, 
there exists a unique global-in-time solution $v$ to~\eqref{NLS2}
with $v|_{t = 0} = u_0$.
Furthermore, 
for each $u_0 \in \Si$, 
the global-in-time solution $v_N$ to~\eqref{NLS1c}
with $v_N|_{t = 0} = u_0$ 
converges to $v$
in $C(\R;  H^{-\eps}(\T^2))$, 
endowed with the compact-open
topology in time, 
as $N \to \infty$.

\end{theorem}

We first note that
 Theorem \ref{THM:1}\,(ii)
follows 
from
 the local-in-time result (Theorem \ref{THM:1}\,(i))
 and Bourgain's invariant measure argument \cite{BO94, BO96}
 together with the invariance of the truncated Gibbs measure
 $\rho_N$ for \eqref{NLS1c};
 see 
\cite{BT2, ORTz}
for details of such an argument. 
Moreover, 
by noting that $\al_N$ in \eqref{C1}
converges $\mu$-almost surely  (and hence $\rho$-almost surely) 
to a finite limit as $N \to \infty$, 
Theorem \ref{THM:0} follows
from 
 Theorem \ref{THM:1}\,(ii) and \eqref{C2}.

\medskip

Our main goal in this note
is to revisit the proof of 
the $\mu$-almost sure local well-posedness of the renormalized NLS \eqref{NLS2}
by using 
tools and ideas 
from the recent development in the field, 
in particular 
the random
tensor estimate (Lemma~\ref{LEM:RT}) developed
by Deng, Nahmod, and Yue~\cite{DNY3}.
Note that  the local-in-time convergence claimed in Theorem \ref{THM:1}\,(i) follows from 
a straightforward modification
of such an argument, 
and hence we omit details.

Before proceeding further, 
we first recall  the known deterministic results
for \eqref{NLS1}.
In \cite{BO93}, Bourgain introduced the Fourier restriction norm method, 
based on the $X^{s, b}$-space defined via the norm:
\begin{align}
\| u \|_{X^{s, b}} = \|\jb{n}^s\jb{\tau +|n|^2}^b \ft u(n, \tau)\|_{\l^2_n L^2_\tau}, 
\label{Xsb}
\end{align}

\noi
and 
proved local well-posedness of \eqref{NLS1}
in $H^s(\T^2)$ for any $s > 0$
by establishing  the following trilinear estimate;
given 
any $s > 0$ and
$b_1 > b_2 > \frac 12$ sufficiently close to $\frac 12$, 
we have 
\begin{align}
\|u_1 \cj u_2 u_3\|_{X^{s, b_1-1}_T} \les  \prod_{j = 1}^3 \|u_j\|_{X^{s, b_2}_T}
\label{tri1}
\end{align}

\noi
for any  $0 < T \le 1$, 
where $X^{s, b}_T$ denotes the local-in-time version of the $X^{s, b}$-space
defined in~\eqref{Xsb2}.
This local well-posedness is essentially sharp
in the sense that, 
when $s < 0$, 
\eqref{NLS1} is known to be ill-posed in $H^s(\T^2)$; 
see~\cite{Oh17}.
A key ingredient for the trilinear estimate~\eqref{tri1}
is the following 
$L^4$-Strichartz estimate %(with a derivative loss) 
due to Bourgain \cite{BO93}:
\begin{align}
\| e^{it \Dl}  \P_N f\|_{L^4([0, 1]; L^4(\T^2))} \les N^s \|\P_N f\|_{L^2}
\label{Str1}
\end{align}

\noi
for any $s> 0$.
 We point out that  \eqref{Str1}
fails when $s = 0$ and 
the well-posedness issue of the cubic NLS
in the critical space $L^2(\T^2)$ remains
a challenging open problem.
See \cite[Remark~4.3]{Oh13} and \cite[Section~3]{Kishimoto}
(see also \cite{TTz})
for a sharp lower bound
and  a recent breakthrough work \cite{HK24} for a sharp upper bound
for the $L^4$-Strichartz estimate on $\T^2$.

Let us turn our attention to the renormalized NLS \eqref{NLS2}.
Define the trilinear operators
$\NN$ and $\RR$
by setting
\begin{align}
\begin{split}
\F_x\big(\NN(v_1,v_2,v_3)\big)(n) 
& = \sum_{\substack{n = n_1 - n_2 + n_3\\ n\neq n_{1},n_{3}}}
\ft v_1 (n_1) \cj{\ft v_2 (n_2)} \ft v_3 (n_3), \\
\F_x\big(\RR(v_1,v_2,v_3)\big) (n) 
& = 
\ft v_1 (n) \cj{\ft v_2 (n)} \ft v_3 (n), 
\end{split}
\label{non1}
\end{align}

\noi
where $\F_x$ denotes the Fourier transform on $\T^2$.
For simplicity, 
we set
$\NN(v) = \NN(v, v, v)$
and $\RR(v) = \RR(v, v, v)$.
Then, from \eqref{NLS2} and \eqref{non1}, we have 
\begin{align}
\Nf(v) =  \NN (v) - \RR (v).
\label{non2}
\end{align}

\noi
The first term $\NN(v)$ represents the non-resonant contribution, 
while 
the second term $\RR(v)$ represents the (doubly) resonant contribution.
We point out that 
the renormalization in $\Nf(v)$
essentially removes the resonant contribution from $n = n_1$ or $n_3$
in the usual cubic nonlinearity (at the expense
of adding the doubly resonant term $\RR(v)$)
but that 
the well-\,/\,ill-posedness
results for \eqref{NLS1} mentioned above
also apply to the renormalized NLS \eqref{NLS2}.
In particular, 
while the Gaussian free field $\mu$ is supported on  $H^s(\T^2) \setminus L^2(\T^2)$, 
$s < 0$, 
the renormalized NLS~\eqref{NLS2} is ill-posed in $H^s(\T^2)$
for any $s < 0$.

In order to overcome this difficulty, 
Bourgain \cite{BO96}
considered the following first order expansion:
\begin{align*}
v(t) = z(t) + w(t),  
\end{align*}

\noi
where 
\begin{align}
z(t)= z^\o(t) = e^{it \Dl} u_0^\o
=  \sum_{n \in \Z^2 } \frac{ g_n(\o)}{\jb{n}} e^{-it|n|^2}e_n
\label{lin1}
\end{align}

\noi
 denotes the random linear solution with $\Law(u_0^\o) = \mu$, 
and studied the following equation satisfied by the residual term $w$:
\begin{align}
\label{NLS3}
\begin{cases}
i \partial_t w + \Delta w = 
\Nf(w+z)\\
w|_{t = 0} = 0.
\end{cases}
\end{align}

\noi
See also \cite{McK, DPD}.
By writing \eqref{NLS3} in the Duhamel formulation, we have
\begin{align*}
w(t) = - i \int_0^t e^{i (t - t') \Dl} \Nf(w+z)(t') dt'.
\end{align*}

\noi
Then, 
in view of 
the nonhomogeneous standard linear estimate (see
\cite{BO93, KPV93, TAO}):
\begin{align*}
\bigg\| \int_0^te^{i(t- t')\Dl} F(t') dt' \bigg\|_{X^{s, \frac 12 + \eps}_T}\les 
T^\eps \|F\|_{X^{s, - \frac 12 + 2\eps}_T}
\end{align*}

\noi
for any sufficiently small $\eps > 0$, 
almost sure local well-posedness of \eqref{NLS3}
(and hence Theorem~\ref{THM:1}\,(i))
follows once we prove the following proposition.\footnote{We also need
a multilinear version of \eqref{tri2}.
Since it follows from a straightforward modification of the proof of \eqref{tri2}, %Proposition \ref{PROP:1}, 
we omit details.}

\begin{proposition}\label{PROP:1}
Let $\Nf(v)$ be as in \eqref{non2}.
Given small   $s > 0$ and $\eps > 0$, 
there exist  $c_1, c_2 > 0$ 
with the following property; 
given  $0 < T \ll 1$, there exists $\Si_T \subset H^{-\eps}(\T^2)$
with $\mu (\Si_T^c) < e^{-\frac{c_1}{ T^{c_2}}}$ such that 
we have 
\begin{align}
\|\Nf (w + z)\|_{X^{s, -\frac 12 +2 \eps}_T} 
\les 1
\label{tri2}
\end{align}

\noi
for any    $w$ with $\|w\|_{X^{s, \frac 12 + \eps}_T} \le 1$
and $u_0 \in \Si_T$, 
where 
% $\Nf(v_1, v_2, v_3) =  \NN (v_1, v_2, v_3) + \RR (v_1, v_2, v_3)$
%and 
$z(t) = e^{it \Dl} u_0$. % denotes the random linear solution with $\Law(u_0) = \mu$.

\end{proposition}

In view of 
\eqref{non2}, 
the multilinearity of $\NN$ and $\RR$, 
and the symmetry in the first and third arguments in~\eqref{non1}, 
it suffices to estimate  the following terms:

\smallskip

\begin{itemize}
\item[(i)]
resonant term: $\RR(w+z)$; see Section \ref{SEC:res},

\smallskip

\item[(ii)]
``deterministic'' term:
$\NN(w, w, w)$; see \eqref{tri1},

\smallskip

\item[(iii)]
purely stochastic term:
$\NN(z, z, z)$;
 see Section \ref{SEC:PS1},

\smallskip

\item[(iv)]
random tensor term I: $\NN(z, w, z)$; see Section \ref{SEC:RT1},

\smallskip

\item[(v)]
random tensor term II: $\NN (w, z, z)$; see Section \ref{SEC:RT2},

\smallskip

\item[(vi)]
random tensor term III: $\NN(z, w,w)$; see Section \ref{SEC:RT3},

\smallskip

\item[(vii)]
random tensor term IV:  $\NN (w, z, w)$; see Section \ref{SEC:RT4}, 

\end{itemize}

\smallskip

\noi
where, with a slight abuse of notation, 
we refer to the terms (iv) - (vii) 
as random tensor terms.
Since the contribution from 
the
``deterministic'' term
$\NN(w, w, w)$
is already handled in~\eqref{tri1}, 
the rest of this note is devoted
to estimating the contributions
from the other terms.

\begin{remark}\rm
We conclude this introduction
by mentioning several works related
to Theorem~\ref{THM:0}.
In \cite{FOSW}, 
Fan, Ou, Staffilani, and Wang
extended 
Theorem \ref{THM:0}
 to the irrational torus case.
In \cite{Zine}, 
Zine established
convergence of invariant Gibbs dynamics
for the stochastic  complex 
Ginzburg-Landau equation on $\T^2$
with the defocusing cubic nonlinearity 
to that for the defocusing cubic NLS on $\T^2$
as the dissipation and the noise vanish
in an appropriate manner.
\end{remark}

\section{Preliminaries}
\label{SEC:2}

\subsection{Notations and basic deterministic estimates}

Let $A\les B$ denote an estimate of the form $A\leq CB$ for some constant $C>0$. We write $A\sim B$ if $A\les B$ and $B\les A$, while $A\ll B$ denotes $A\leq c B$ for some small constant $c> 0$. 
We may write  $\les_{\al}$ and $\sim_{\al}$ to 
emphasize the dependence on an external parameter $\al$.
We use $C>0$ to denote various constants, which may vary line by line.

Given dyadic $N\ge 1$, 
we use $\Q_N$ to denote the 
Littlewood-Paley projector onto the frequencies
$\{|n| \sim N\}$ for $N \ge 2$ (and $\{|n| \les 1\}$ when $N = 1$).
We then set
\begin{align}
\Sb_N f= \sum_{\substack{1 \le M \le N\\\text{dyadic}}}\Q_Mf
\qquad \text{and}
\qquad 
\Sb_N^\perp f = f - \Sb_N f.
\label{LP2}
\end{align}

\noi
Given dyadic numbers $N_1, N_2, N_3 \ge 1$, 
we use $N_{\max}$, $N_{\med}$, and $N_{\min}$
to denote their decreasing rearrangement:
\begin{align}
N_{\max} \ge N_{\med} \ge N_{\min}.
\label{ord1}
\end{align}

We use $\eta \in C^\infty_c(\mathbb{R})$ to denote a smooth cutoff function supported on $[-2, 2]$ with $\eta \equiv 1$ on $[-1, 1]$,  and let $\eta_{_T}(t) =\eta(T^{-1}t)$.
Then, we have
\begin{equation}
 \|\ft \eta_{_{T}}\|_{L^q_\tau} \sim T^\frac{q-1}{q} \|\ft \eta\|_{L^q_\tau} 
\sim  T^\frac{q-1}{q}
\label{decay}
\end{equation} 

\noi
for any $1 \le q \le \infty$
and $T > 0$.

Given $T > 0$, we define the space $X_T^{s,b}$ to be the restriction of 
the $X^{s, b}$-space defined in~\eqref{Xsb} onto the time interval 
$[-T, T]$ via the norm:
\begin{equation}
\| u \|_{X_T^{s,b}} := \inf \big\{ \| v \|_{X^{s,b}}: v|_{[-T, T]} = u \big\}.
\label{Xsb2}
\end{equation}

\noi
Note that $X^{s, b}_T$ is complete.
Given any $s \in \R$ and $b > \frac 12$, we have $X_T^{s,b} \subset C([-T, T]; H^s(\T^2))$.
Given any $s, b \in \R$, 
the following homogeneous linear estimate holds:
\begin{align}
\| e^{it \Dl} f\|_{X^{s, b}_T}
\les_b \| f \|_{H^s}
\label{lin2}
\end{align}

\noi
for any $0 < T \le 1$.
We also recall the following lemma;
see \cite[Lemma 2.11]{TAO}.

\begin{lemma}
\label{LEM:decay1} Let $s \in\R$ and $- \frac 12 < b_1 \le b_2 < \frac 12$.
Then, we have 
\begin{equation*}
 \|u\|_{X^{s, b_1}_T}  \leq C T^{b_2 - b_1} \|u\|_{X^{s, b_2}_T}
\end{equation*}

\noi
for any $0 < T \le 1$. 

\end{lemma}

In the remaining part of this note, 
we fix a sufficiently small constant $\eps > 0$
and set 
\begin{align}
b = \frac 12 + \eps\qquad \text{and}\qquad 
b' = \frac 12 - 2\eps
\label{b1}
\end{align}

\noi
(unless otherwise specified).
With this notation, we state a slightly improved
trilinear estimate.
Given $u_j$ on $\T^2 \times \R$, $j = 1, 2, 3$, 
suppose that there exists $j_* \in \{1, 2, 3\}$ such that, 
given any $t \in \R$,  
we have 
\[ |n_{j_*}| \ges \max_{j \ne j_*}|n_j|\]

\noi
for any $n_k \in \supp \ft u_k(\cdot, t)$,  $k = 1, 2, 3$.
Then, by a slight modification of the proof of 
the trilinear estimate \eqref{tri1}, we have 
\begin{align}
\| \NN(u_1, u_2, u_3)\|_{X^{s, - b' + 2\eps_1 }} \les 
\|u_{j_*}\|_{X^{s, b-\eps_1}}
 \prod_{\substack{j = 1 \\j \ne j_*}}^3 \| u_j\|_{X^{\eps_2, b - \eps_1}}
\label{tri3}
\end{align}

\noi
for any    $s > 0$  and $ \eps_2 > 0$, 
provided that $\eps_1 \ge 0$ is sufficiently small.

\medskip

Lastly, we recall the following elementary calculus
lemma; see, for example,  \cite[Lemma~4.2]{GTV}.

\begin{lemma}
\label{LEM:conv}

Let  $\al, \be \in \R$ satisfy
\begin{align*}
 \al \ge \be \ge 0 \qquad \text{and}\qquad   \al+ \be > 1.
% \label{SUM1}
\end{align*}

\noi
Then, we have
\begin{align*}
\int_\R\frac{dx}{\jb{x}^\al \jb{x-a}^\be}
& \les \frac 1{\jb{a}^{ \be - \ld}},\\ 
 \sum_{n \in \Z} \frac{1}{\jb{n}^\al \jb{n - a}^\be}
& \les \frac 1{\jb{a}^{ \be - \ld}}
\end{align*}

\noi
for any $a \in \R$,
where $\ld = 
\max( 1- \al, 0)$ when $\al\ne 1$ and $\ld = \eps$ when $\al = 1$ for any $\eps > 0$.
%In particular, if \eqref{SUM1} holds, then we have 
%\[
%\sup_{\xi  \in \R} \int_{\xi = \xi_1 + \xi_2}\frac{d\xi_1}{\jb{\xi_1}^\al \jb{\xi_2}^\be}
%< \infty. \]

\end{lemma}

\subsection{Tools from stochastic analysis}

We first recall a basic bound
on Gaussian random variables;
see \cite[Lemma 3.4]{CO}
for a proof.

\begin{lemma}
\label{LEM:prob1} 

Let $\left\{g_n\right\}_{n\in \Z^d}$ be a sequence of independent standard complex-valued Gaussian
random variables. 
Given any $\dl_1, \dl_2 > 0$, there
exist $c_1, c_2 > 0$ such that 
\begin{equation*}
|g_n(\omega)| \les T^{-\dl_1} \jb{n}^{\dl_2}
\end{equation*}
	
\noi 
for any $n \in \Z^d$ and $0 < T \ll 1$, 
outside an exceptional set of probability $< e^{-\frac{c_1}{T^{c_2}}}$. 
\end{lemma}

Next, we recall  the Wiener chaos estimate
(\cite[Theorem~I.22]{Simon});
see \cite[Lemma 3.2]{OTz1}
for the following particular version.

\begin{lemma}\label{LEM:hyp}
Let $\left\{g_n\right\}_{n\in \Z^d}$ be a sequence of independent standard real-valued Gaussian
random variables. Given $k\in \N$, let $\left\{P_j\right\}_{j\in \N}$ be a sequence of polynomials in 
$\bar{g}=\left\{g_n\right\}_{n\in \Z^d}$ of degree at most $k$. Then, for any finite $p\ge 1$, we have
\begin{align*}
\bigg\|\sum_{j\in \N}P_j(\bar{g})\bigg\|_{L^p(\Omega)}
\le (p-1)^{\frac{k}{2}}\bigg\|\sum_{j\in \N}P_j(\bar{g})\bigg\|_{L^2(\Omega)}.
\end{align*}
\end{lemma}

\subsection{Random tensor estimate}

In this section, we provide the basic definition of tensors and 
state the random tensor estimate (Lemma \ref{LEM:RT}).
See~\cite[Sections 2 and 4]{DNY3}, \cite[Section 4]{Bring1}, 
and \cite[Appendix C]{OWZ}
for further discussion.

\begin{definition} \label{DEF:tensor} \rm
Let $A$ be a finite index set. We denote by $n_A$ the tuple $ \{n_j \}_{j \in A}$. 
 A tensor $h = h_{n_A}$ is a function: $(\Z^d)^{A} \to \mathbb{C} $ with the input variables $n_A$. Note that the tensor $h$ may also depend on $\o \in \O$. 
 The support of a tensor $h$ is the set of $n_A$ such that $h_{n_A} \neq 0$. 

Given a finite index set  $A$, 
let $(B, C)$ be a partition of $A$. We define the norms 
 $\| \cdot \|_{n_A}$ and 
$\| \cdot \|_{n_{B} \to n_{C}}$ by 
\[ \| h \|_{n_A}  = \|h\|_{\l^2_{n_A}} = \bigg(\sum_{n_A} |h_{n_A}|^2\bigg)^\frac{1}{2}\]
and
\begin{align}
  \| h \|^2_{n_{B} \to n_{C}} = \sup \bigg\{ 
\sum_{n_{C}} \Big| \sum_{n_{B}} h_{n_A} f^\textup{in}_{n_{B}} \Big|^2 :  \| f^\textup{in} \|_{\l^2_{n_{B}}} =1  \bigg\},  
\label{Z0a}
\end{align}

\noi
where  we used the short-hand notation $\sum_{n_Z} = \sum_{n_Z \in (\Z^d)^Z}$ for a finite index set $Z$.
Note that, by duality, we have  $\| h \|_{n_{B} \to n_{C}} = \| h \|_{n_{C} \to n_{B}} 
= \| \cj h \|_{n_{B} \to n_{C}}$ for any tensor $h = h_{n_A}$. 
If $B = \varnothing$ or $C = \varnothing$,  then we have
$  \| h \|_{n_{B} \to n_{C}} = \| h \|_{n_A}$.
\end{definition}

Let $(B, C)$ be a partition of $A$.
Then, 
by duality, we can write \eqref{Z0a} as 
\begin{align}
  \| h \|_{n_{B} \to n_{C}} = \sup \bigg\{ 
\Big|  \sum_{n_{B}, n_C} h_{n_A} f^\textup{in}_{n_{B}} f^\textup{out}_{n_C}\Big| : 
\| f^\textup{in} \|_{\l^2_{n_{B}}} =  \| f^\textup{out} \|_{\l^2_{n_{C}}} =1   \bigg\}.
\label{Z0x}
\end{align}
%
%\noi
%from which we obtain
%\begin{align}
%\sup_{n_A}|h_{n_A}|
%= \sup_{n_B, n_C}|h_{n_Bn_C}|
%\le   \| h \|_{n_{B} \to n_{C}}.
%\label{Z0b}
%\end{align}

For example, when $A = \{1, 2\}$, 
the norm  $\| h \|_{n_{1} \to n_{2}}$ denotes the usual operator norm
$\| h \|_{\l^2_{n_{1}} \to \l^2_{n_{2}}}$
for an infinite-dimensional matrix operator $\{h_{n_1 n_2}\}_{n_1, n_2 \in \Z^d}$.
By bounding the matrix operator norm by the Hilbert-Schmidt norm (= the Frobenius norm), we have
\begin{align}
\| h \|_{\l^2_{n_{1}} \to \l^2_{n_{2}}} \le \| h\|_{\l^2_{n_1, n_2}}.
\label{Z0}
\end{align}

With the above notation, we have the following random tensor estimate,
originally introduced  in \cite[Propositions 4.14 and 4.15]{DNY3}; see also  \cite[Proposition 4.50]{Bring1} and \cite[Lemma C.3]{OWZ} for similar formulations as ours.

\begin{lemma}\label{LEM:RT}
Let $A$ be a finite index set with $k = |A| \ge 1$.
Given a \textup{(}deterministic\textup{)} tensor
 $h_{bcn_A} \in \l^2_{bcn_A}$
 such that $n_a \in \Z^d$ for each $a \in A$ and $(b,c) \in (\Z^d)^q$ for some integer $q \geq 2$, define 
a random tensor $H = H_{bc}$ by
\begin{align}
H_{bc} =  \sum_{n_A} h_{bcn_A} 
\prod_{n \in \Z^d} 
\Hf_{k_n}(g_{n}), 
\label{RTX0}
\end{align}

\noi
where $\{g_n\}_{n \in \Z^d}$ is a family of independent
standard real-valued Gaussian random variables
and 
$\Hf_{k_a}$ denotes the Hermite polynomial of degree\footnote{Since $A$ is 
a finite set, we have $k_n = 0$ except for finitely many $n$'s.
Recalling that $\Hf_0(x) = 1$, we see that the product in \eqref{RTX0}
is indeed over a finite set $\{n \in \Z^d:\,\text{given $n_A \in (\Z^d)^A$, there exists } a \in A \text{ such that }n_a = n \}$.}
\begin{align}
k_n = \#\{ a\in A: n_a = n\}
\label{RTX0a}
\end{align}
such that $\sum_{n \in \Z^d} k_n = k$.

\medskip

\noi
\textup{(i)}
Given $N \in \N$, suppose that 
\begin{align*}
\supp h \subset \big\{ |b - b^*|, |c - c^*|, |n_a| \les N
\text{ for each $a \in A$}
 \big\}
\end{align*} 

\noi
for some fixed $(b^*, c^*) \in (\Z^d)^q$.
Then, 
 for any $\theta > 0$ and finite $p \ge 1$, we have
\begin{align*}
\big\| \| H_{bc} \|_{b \to c} \big\|_{L^p(\O)} 
\les p^{\frac k2} N^{\theta}  \max_{(B,C)} \| h \|_{b n_B \to c n_C}, 
%\label{RTX0c}
\end{align*}

\noi
where the maximum is taken over  all partitions $(B,C)$ of $A$
and the implicit constant is independent of $N \in \N$
and $(b^*, c^*) \in (\Z^d)^q$.

\medskip

\noi
\textup{(ii)}
Write $b = (b_0, b_1) \in \Z^d \times (\Z^d)^{q_1}$
and
$c = (c_0, c_1) \in \Z^d \times (\Z^d)^{q_2}$
with $q_1 + q_2 + 2 = q$.
Given $N \in \N$, suppose that

\smallskip
\begin{enumerate}
\item[(ii.a)]
$\supp h \subset \big\{ |b_1 - b^*|, |c_1 - c^*|, |n_a| \les N
\text{ for each $a \in A$}\big\}$
for some fixed $(b^*, c^*) \in (\Z^d)^{q_1 + q_2}$,

\smallskip

\item[(ii.b)]
 $|b_0-c_0| \le C_1 N$ and $\big||b_0|^2 -|c_0|^2\big| \les N^{a_1}$ for some $C_1, a_1 \ge 2$, 

\smallskip
\item[(ii.c)] the tensor $h_{bcn_A}
= h_{(b_0, b_1), (c_0, c_1), n_A}$  depends only on $b_0-c_0$, $|b_0|^2 -  |c_0|^2$, 
$b_1$, $c_1$, and $n_A$. 
\end{enumerate}

\noi 

\smallskip

\noi
Given $m \in \Z^d$, define a tensor $h^m_{(b_0', b_1), (c_0', c_1), n_A} $ 
by 
\begin{align}
h^m_{(b_0', b_1), (c_0', c_1), n_A} 
= h_{(b_0'+m, b_1),(c_0'+m, c_1), n_A} \cdot \ind_{|b'_0|, |c_0'| \le C_1 N}.
\label{XY2}
\end{align}

\noi
Then,  
 for any $\theta > 0$ and finite $p \ge 1$, we have
\begin{align}
\big\| \| H_{bc} \|_{b \to c} \big\|_{L^p(\O)} 
\les p^{\frac k2} N^{\theta}
\sup_{m \in \Z^d}
 \max_{(B,C)} \| h^m_{(b_0', b_1), (c_0', c_1), n_A}  \|_{(b_0', b_1), n_B \to (c_0', c_1),  n_C}, 
\label{RTX0c}
\end{align}

\noi
where 
the maximum is taken over  all partitions $(B,C)$ of $A$
and 
the implicit constant is independent of $N \in \N$
and $(b^*, c^*) \in (\Z^d)^{q_1 + q_2}$.

\medskip

\noi
\textup{(iii)}
Let $h$ be as in Part (ii), and set 
\begin{align*}
\wt h_{(b_0, b_1), (c_0, c_1), n_A} 
= h_{(b_0, b_1), (c_0, c_1), n_A} 
\cdot 
 \ind_{|b_0|, |c_0| \ge 10C_1 N}, 
%\label{XY1}
\end{align*}

\noi
where $C_1 \ge 1$ is as in 
(ii.b) above, 
and 
 define  a random  tensor $\wt H = \wt H_{bc}$ by
\begin{align*}
\wt H_{bc} =  \sum_{n_A} \wt h_{bcn_A} 
\prod_{n \in \Z^d} 
\Hf_{k_n}(g_{n}), 
\end{align*}

\noi
where
$\{g_n\}_{n \in \Z^d}$ and $k_n$ are as in \eqref{RTX0}.
Given $m \in \Z^d$, 
define a tensor $\wt h^m_{(b'_0, b_1), (c_0', c_1), n_A} $ by 
\begin{align}
\begin{split}
\wt h^m_{(b'_0, b_1), (c_0', c_1), n_A} 
& = \wt h_{(b'_0+m, b_1),(c_0'+m, c_1), n_A} \cdot \ind_{|b'_0|, |c'_0| \le C_1 N}\\
& =  h_{(b'_0+m, b_1),(c_0'+m, c_1), n_A} 
\cdot  \ind_{|b_0'+m|, |c_0'+m| \ge 10C_1 N} 
\cdot \ind_{|b'_0|, |c'_0| \le C_1 N}.
\end{split}
\label{XY4}
\end{align}

\noi
Then, 
 for any $\theta > 0$ and finite $p \ge 1$, we have
\begin{align}
\big\| \| \wt H_{bc} \|_{b \to c} \big\|_{L^p(\O)} 
\les p^{\frac k2} N^{\theta}
\sup_{m \in \Z^d}
 \max_{(B,C)} \|
\wt h^m_{(b'_0, b_1), (c_0', c_1), n_A} \|_{(b_0', b_1), n_B \to (c_0', c_1),  n_C}, 
\label{XY1a}
\end{align}

\noi
where the maximum is taken over  all partitions $(B,C)$ of $A$
and the implicit constant is independent of $N \in \N$
and $(b^*, c^*) \in (\Z^d)^{q_1 + q_2}$.

\end{lemma}

In \cite{OWZ}, 
the random tensor estimate 
was 
extended to multiple stochastic integrals (with respect to white-in-time noises).
In \cite{CLO, COZ}, 
the random tensor estimate
was further extended 
for multiple stochastic integrals with respect to fractional-in-time noises.
These extensions have played a crucial 
role
in establishing pathwise well-posedness
of stochastic dispersive PDEs with multiplicative noises;
see also \cite{CGLLO}.

See 
\cite[Lemma C.3 and Remark C.4 (ii)]{OWZ}
for a proof of Lemma \ref{LEM:RT}\,(i). 
When $b = b_0$ and $c = c_0$, 
Lemma~\ref{LEM:RT}\,(ii) is essentially proven in \cite[Proposition 4.15]{DNY3}.
While the general case follows from a straightforward modification, 
we include its proof for readers' convenience.
For this purpose, we need the following
compactness lemma.
See also 
\cite[Claim~5.2]{DNY2},
 \cite[Claim~3.7]{DH1}, and \cite[Claim~4.16]{DNY3}.

\begin{lemma} \label{LEM:DH}

Let $a_1 \ge 1 $.
Given $N \gg 1$, 
let $\{f_m\}_{m \in \Z^d}$ be the family of functions defined by
\begin{align*}
f_m (n) = m \cdot n
\end{align*}

\noi
with the domain given by 
\begin{align}\label{EJ1x}
D_m = \big\{n \in \Z^d: |n| \les N \textup{ and } |m \cdot n| \les N^{a_1}\big\}.
\end{align}
Then, there exists $a_2 \gg 1$, independent of $N$,  and $J = J (a_1,a_2,N) \subset \Z^d$ 
with  $|J| = O(N^{a_2})$
such that, 
given  any $m \in \Z^d$,
there exists $m_0 \in J$ such that 
\begin{align}
f_{m_0} = f_m\qquad \text{and}\qquad  D_{m_0} = D_m.
\label{EJ2y}
\end{align}
\end{lemma}

\begin{proof}
We first consider  the one-dimensional case.
For $|m| \les N^{a_1}$, we include such $m$ in $J$.
If $|m| \gg N^{a_1}$, 
it follows from \eqref{EJ1x} that 
$D_m = \{0\}$ and thus $f_m \equiv 0$ on $D_m$.
Thus, by including one such $m_0$ with 
 $|m_0| \gg N^{a_1}$
in $J$, 
we have 
 $f_{m} = f_{m_0}\equiv 0$ 
with $D_{m} = D_{m_0} = \{0\}$ for any $|m|\gg N^{a_1}$.
By construction, we have 
 $|J| = O(N^{a_1})$ 
such that \eqref{EJ2y} holds.

Next, we consider the higher-dimensional case.
Fix $m \in \Z^d$, and write 
\begin{align}
 D_m = \bigcup_{|p| \les N^{a_1}} D_{m, p}, 
\quad \text{where}\ \  D_{m, p} = D_m \cap \{m\cdot n  = p \}.
\label{EJ2a}
\end{align}

\noi
Fix 
$p$ with $|p| \les N^{a_1}$
such that  $D_{m, p} \ne \varnothing$, 
and 
let $\{n_i\}_{i = 0}^r \subset D_{m, p}$
(for some $0 \le r \le d$) be a maximal affine independent\footnote{Namely, 
there exists no $(\al^0, \al^1, \dots, \al^r) \subset \R^{r+1}\setminus\{0\}$ with $\sum_{i = 0}^r \al^i = 0$
such that $\sum_{i = 0}^r \al^i n_i = 0$.} set in $D_{m, p}$.
Let $L$ be the lattice generated by 
$\{n_i - n_0\}_{i = 0}^r$, 
and fix an 
 LLL reduced basis; see, for example, 
 \cite{LLL} and \cite[Chapter 4]{Bre}
 for the definition of an LLL reduced basis
$\{\l_i\}_{i = 1}^r$.
Then, it follows from 
\cite[Theorem 4.8]{Bre}
that 
\begin{align*}
|\l_i| \les \max_{k = 1, \dots,  r}|n_k - n_0|\les N.
\end{align*}

\noi
Then, given any $n \in D_{m, p}$, there is a unique  vector 
$\al = \al(n) = (\al^1, \dots, \al^r) \in \Z^r$ 
such that $|\al| \les N$,\footnote{This follows
from the almost orthogonality of the LLL reduced basis $\{\l_i\}_{i = 1}^r$;
see \cite[p.\,56]{Bre}.}
$n - n_0 = \sum_{i=1}^r \al^i \l_i$, and 
\begin{align}
m \cdot n = m \cdot \bigg( \sum_{i=1}^r \al^i \l_i\bigg) + m \cdot n_0 
= y \cdot \al + m \cdot n_0,
\label{EJ4}
\end{align}
%g_j (n) = 
where $y = (y^1, \cdots, y^r) \in \Z^r$ with 
$y^i = m \cdot \l_i \in \Z$.
Since 
$n, n_0 \in D_{m, p} \subset D_m$, 
it follows from~\eqref{EJ1x}
that 
 $|m \cdot n|, |m \cdot n_0| \les N^{a_1}$.
Then, using \eqref{EJ4}, we see that 
 $|y \cdot \al(n)| \les N^{a_1}$
 for any $n \in D_{m, p}$.
For each $i = 1, \dots, r$, 
define $\al_i = (\al_i^1, \dots, \al_i^r)$
such that 
 $n_i - n_0 = \sum_{i=1}^r \al^k_i \l_k$.
Since
$\{\al_i\}_{i = 1}^r$ is a set of  linearly independent  vectors 
in $\Z^r$ 
with $|\al_i| = O(N)$, 
we conclude that $|y| \les N^{Ca_1}$; see Lemma \ref{LEM:Y1}
and Remark \ref{REM:Y2}.

Given $m \in \Z^d$
and $p$ with $|p| \les N^{a_1}$, 
we define $F_{m, p}$ on $D_{m, p}$ by 
\begin{align}
 F_{m, p}(n) = m\cdot n = y\cdot \al + m\cdot n_0, \quad n \in D_{m, p}, 
 \label{EJ5}
\end{align}

\noi
where $y$ and $\al = \al(n)$ are as in \eqref{EJ4}.
Since $D_{m, p}$ is generated by 
$\{n_i\}_{i = 0}^r$, 
it follows from the linearity that 
the function 
$F_{m, p}$  is completely
determined by its action
on $\{n_i\}_{i = 0}^r$.
Hence, by recalling that 
$|y| \les N^{Ca_1}$, $|\al(n)|\les N$
for any $n \in D_{m, p}$, 
and 
$|m \cdot n_0| \les N^{a_1}$, 
it follows from~\eqref{EJ5} that 
the collection $\{ F_{m, p}\}_{m \in \Z^d}$
consists of $O(N^{a_1'})$ many elements for some $a_1' > a_1$.
Now, we define $F_m$ on $D_m$ by 
setting $F_m = F_{m, p}$ on $D_{m, p}$
for each  $p$ with $|p| \les N^{a_1}$.
Then, 
from~\eqref{EJ2a}, we see  that 
the collection $\{ F_{m}\}_{m \in \Z^d}$
consists of $O(N^{a_2})$ many elements for some $a_2 > a_1'$.
Finally, denoting $\{ F_{m}\}_{m \in \Z^d}$
by $\{f_{m_0}\}_{m_0 \in J}$ 
for some set $J \subset\Z^d$ with $|J| = O(N^{a_2})$
(where $f_{m_0} (n) = m_0\cdot n$ with the domain  $D_{m_0}$
defined in \eqref{EJ1x}), 
we conclude that~\eqref{EJ2y} holds.
\end{proof}

We now discuss a proof of  Lemma \ref{LEM:RT}\,(ii)
and (iii).

\begin{proof}
[Proof of Lemma \ref{LEM:RT} (ii) and (iii)]

On the support of $h$, we have $|b_0- c_0|\les N$. 
We first restrict $h$ to  $|b_0'|, |c_0'| \les N$, 
where  $b_0' = b_0-m$ and $c_0' = c_0-m$ (for suitable $m \in \Z^d$).
More precisely, we
 set 
\begin{align}
H^m_{(b'_0, b_1),  (c'_0, c_1)} = \sum_{n_A} h^m_{(b'_0, b_1), (c_0', c_1), n_A} 
\prod_{n \in \Z^d} 
\Hf_{k_n}(g_{n}), 
\label{RTX1}
\end{align}

\noi
where $h^m$ and $k_n$ are as in \eqref{XY2} and \eqref{RTX0a}, 
respectively.
For fixed $m \in \Z^d$, 
 Lemma \ref{LEM:RT}\,(i) yields
\begin{align}
\begin{split}
& \big\| \| H_{(b'_0, b_1), (c_0', c_1)}^m \|_{(b_0', b_1) \to (c_0', c_1)} \big\|_{L^p(\O)} \\
& \quad \les p^{\frac k2} N^{\frac \theta 2}
 \max_{(B,C)} \|
h^m_{(b'_0, b_1), (c_0', c_1), n_A} \|_{(b_0', b_1), n_B \to (c_0', c_1),  n_C}, 
%& \quad \le p^{\frac k2} N^{\frac \theta2} \max_{(B,C)} \| h \|_{b n_B \to c n_C}, 
\end{split}
\label{RTX2}
\end{align}

\noi
where the implicit constant is independent of $m \in \Z^d$.
By the hypothesis (ii.c), 
the tensor $h^m_{(b_0', b_1), (c_0', c_1), n_A}$ 
is determined by $(b_0', b_1, c_0', c_1, n_A)$ and 
\[|b_0'+m|^2 - |c_0'+m|^2 = 2m \cdot (b_0'-c_0') + (|b_0'|^2 - |c_0'|^2).\]

\noi 
From Lemma \ref{LEM:DH}, 
we see that there exists $\{m_j\}_{j=1}^J \subset \Z^d $ with $J =O(N^{a_2})$ 
such that for any $m \in \Z^d$, 
 there exists $j \in \{1, \dots, J\}$ such that 
\begin{align}
|b_0'+m|^2 - |c_0'+m|^2 = |b_0'+m_j|^2 - |c_0'+m_j|^2 
\label{XY2a}
\end{align}

\noi
for any  $b_0',c_0' \in \Z^d $ such that $|b_0'|, |c_0'| \le C_1 N$ and 
$\big||b_0'+m|^2 - |c_0'+m|^2\big| \les N^{a_1}$, 
(which in particular implies 
$|m \cdot (b_0'-c_0')|\les N^{a_1}$).
% By orthogonality and \eqref{RTX2}, we may assume that different $f_i$ are $N$ apart. 

Fix small $\ta > 0$.
Then, by almost orthogonality, the observation above
with $J =O(N^{a_2})$,  and~\eqref{RTX2}, we have
\begin{align}
\begin{split}
\big\| \| H_{bc} \|_{b \to c} \big\|_{L^p(\O)}  
& \les \Big\| \sup_{m \in \Z^d} 
 \| H_{(b'_0, b_1), (c_0', c_1)}^m \|_{(b_0', b_1) \to (c_0', c_1)}  \Big\|_{L^p(\O)} \\
& =  \Big\| \sup_{j = 1, \dots, J} 
 \| H_{(b'_0, b_1), (c_0', c_1)}^{m_j} \|_{(b_0', b_1) \to (c_0', c_1)} 
 \Big\|_{L^p(\O)} \\
& \le
  \bigg\| \Big(\sum_{j = 1, \dots, J} 
   \| H_{(b'_0, b_1), (c_0', c_1)}^{m_j} \|_{(b_0', b_1) \to (c_0', c_1)}^p 
 \Big)^\frac 1p\bigg\|_{L^p(\O)} \\
& \les
N^{\frac{a_2}{p}}
\sup_{j = 1, \dots, J}
  \big\|  
   \| H_{(b'_0, b_1), (c_0', c_1)}^{m_j} \|_{(b_0', b_1) \to (c_0', c_1)}  \big\|_{L^p(\O)} \\
 & \les
 p^{\frac k2} N^{\theta}
 \sup_{m \in \Z^d}
 \max_{(B,C)} \| h^m_{(b_0', b_1), (c_0', c_1), n_A}  \|_{(b_0', b_1), n_B \to (c_0', c_1),  n_C}
%
% \Big( \max_{(B,C)} \| h \|_{b n_B \to c n_C}\Big)
\end{split}
\label{RTX3}
 \end{align}

\noi
for any $p > \frac{2a_2}{\ta} \gg 1$.
This proves
\eqref{RTX0c}.

\medskip

\noi
(iii)
As in Part (ii), with $b_0' = b_0-m$ and $c_0' = c_0-m$, 
we define 
\begin{align}
\wt H^m_{(b'_0, b_1),  (c'_0, c)} = \sum_{n_A} \wt h^m_{(b'_0, b_1), (c_0', c_1), n_A} 
\prod_{n \in \Z^d} 
\Hf_{k_n}(g_{n}), 
\label{XY3}
\end{align}

\noi
where $\wt h^m$ and $k_n$ are as in \eqref{XY4} and \eqref{RTX0a}, respectively.
Then, 
it follows from 
 Lemma \ref{LEM:RT}\,(i) that 
\begin{align}
\begin{split}
&  \big\| \| \wt H_{(b'_0, b_1), (c_0', c_1)}^m \|_{(b_0', b_1) \to (c_0', c_1)} \big\|_{L^p(\O)} \\
& \quad 
 \les p^{\frac k2} N^{\frac \theta2}
  \max_{(B,C)} \|
\wt h^m_{(b'_0, b_1), (c_0', c_1), n_A} \|_{(b_0', b_1), n_B \to (c_0', c_1),  n_C}, 
\end{split}
\label{XY5}
\end{align}

\noi
where the implicit constant is independent of $m \in \Z^d$; 
see \eqref{RTX2}.

We first consider the case $|m| \ge 20 C_1N$.
In this case, 
we have 
$|b_0'+m|, |c_0'+m| \ge 19 C_1 N $
and thus the condition 
$|b_0'+m|, |c_0'+m| \ge 10C_1 N$
on the right-hand side of  \eqref{XY4} is vacuous.
In particular, 
by comparing \eqref{XY4} with \eqref{XY2}, we
note  that 
the tensor $\wt h^m_{(b_0', b_1), (c_0', c_1), n_A}$ 
%defined in \eqref{XY4}
depends only on  $(b_0',b_1, c_0', c_1, n_A)$ and 
$|b_0'+m|^2 - |c_0'+m|^2$.
Then, proceeding as in Part~(ii) with 
Lemma~\ref{LEM:DH}, 
we see that there exists $\{m_j\}_{j=1}^J \subset \Z^d $ with $J =O(N^{a_2})$ 
such that for any $m \in \Z^d$
with $|m| \ge 20 C_1N$, 
 there exists $j \in \{1, \dots, J\}$ such that \eqref{XY2a} holds
for any  $b_0',c_0' \in \Z^d $ such that $|b_0'|, |c_0'| \le C_1 N$ and 
$\big||b_0'+m|^2 - |c_0'+m|^2\big| \les N^{a_1}$.
Without loss of generality, 
we assume that $|m_j| \ge 20 C_1 N$ for each $j = 1, \dots, J$, 
since, by definition, for each $j = 1, \dots, J$,
\eqref{XY2a} holds for some $|m|\ge 20 C_1 N$
and thus, if $|m_j| <  20 C_1 N$, 
we can replace $m_j$ by such $m$.

%(which in particular implies 
%$|m \cdot (b_0'-c_0')|\les N^{a_1}$).
% By orthogonality and \eqref{RTX2}, we may assume that different $f_i$ are $N$ apart. 

Fix small $\ta > 0$.
Proceeding as in \eqref{RTX3}
with \eqref{XY5}, we then obtain
\begin{align}
\begin{split}
\big\| \| \wt H_{bc} \|_{b \to c} \big\|_{L^p(\O)}  
& \les \Big\| \sup_{m \in \Z^d} 
 \| \wt H_{(b'_0, b_1), (c_0', c_1)}^m \|_{(b_0', b_1) \to (c_0', c_1)}  \Big\|_{L^p(\O)} \\
& \le \Big\| \sup_{|m| < 20 C_1 N}
 \| \wt H_{(b'_0, b_1), (c_0', c_1)}^m \|_{(b_0', b_1) \to (c_0', c_1)}  \Big\|_{L^p(\O)} \\
& \quad 
+  \Big\| \sup_{|m| \ge  20 C_1 N} 
 \| \wt H_{(b'_0, b_1), (c_0', c_1)}^m \|_{(b_0', b_1) \to (c_0', c_1)}  \Big\|_{L^p(\O)} \\
& \les
N^{\frac{d}{p}}
 \sup_{|m| < 20 C_1 N}
  \big\|  
   \| \wt H_{(b'_0, b_1), (c_0', c_1)}^{m} \|_{(b_0', b_1) \to (c_0', c_1)}  \big\|_{L^p(\O)} \\
& \quad + 
N^{\frac{a_2}{p}}
\sup_{j = 1, \dots, J}
  \big\|  
   \| \wt H_{(b'_0, b_1), (c_0', c_1)}^{m_j} \|_{(b_0', b_1) \to (c_0', c_1)}  \big\|_{L^p(\O)} \\
 & \les
 p^{\frac k2} N^{\theta}
\sup_{m \in \Z^d}
  \max_{(B,C)} \|
\wt h^m_{(b'_0, b_1), (c_0', c_1), n_A} \|_{(b_0', b_1), n_B \to (c_0', c_1),  n_C}
\end{split}
\label{XY6}
 \end{align}

\noi
for any $p >  \max\big( \frac d\ta,  \frac{2a_2}{\ta}\big) \gg 1$, 
where we applied Part (i) of this lemma in 
estimating the contribution from $|m|< 20 C_1 N$
(since 
we have 
 $|b_0|, |c_0| < 21 C_1 N$ under 
$|m|< 20 C_1 N$).
This proves~\eqref{XY1a}.
\end{proof}

\begin{remark}
\rm 

(i)
Instead of the hypothesis (ii.c)
Lemma \ref{LEM:RT}\,(ii), 
 suppose that 
\begin{itemize}

\smallskip

\item[(ii.c')] except for  $O(N)$ many values of $m \in \Z^d$, 
$h^m$
defined in  \eqref{XY2}
is determined by $(b_0', b_1, c_0', c_1, n_A)$ and 
$|b_0'+m|^2 - |c_0'+m|^2$.

\smallskip

\end{itemize}

\noi
Then, 
proceeding as in the proof of Lemma \ref{LEM:RT}\,(iii), 
we see that the conclusion of 
Lemma~\ref{LEM:RT}\,(ii) holds.

\smallskip

\noi
(ii)
%From By a simple bound, applying 
%Lemma 4.10 in \cite{DNY3} to 
%the right-hand side of 
From \eqref{RTX0c} with \eqref{Z0x}, 
we have 
\begin{align*}
\big\| \| H_{bc} \|_{b \to c} \big\|_{L^p(\O)} 
\les p^{\frac k2} N^{\theta}
\max_{(B,C)} \| h \|_{b n_B \to c n_C},
\end{align*}

\noi
which is closer to the format of 
\cite[Proposition 4.15]{DNY3}.
%We, however, stated 
%\eqref{RTX0c} as it is, 
%since it is  
%
%
%
A similar comment applies to 
Lemma~\ref{LEM:RT}\,(iii).

\end{remark}

\begin{remark}\label{REM:tensor}\rm

As it is stated, the random tensor estimate (Lemma \ref{LEM:RT})
is  for linear operators.
However, given a random multilinear operator of the form: %$T(u_1, \dots, u_k)$, 
%we can view it as a linear operator
%acting on the tensor product $u_1\otimes \cdots \otimes u_k$
%and apply Lemma \ref{LEM:RT}.
\begin{align*}
H_{b_1, \dots, b_k, c}
= 
\sum_{n_A} h_{b_1, \dots, b_k,c, n_A} 
\prod_{n \in \Z^d} 
\Hf_{k_n}(g_{n}), 
\quad 
b_1, \dots, b_k \in \Z^d, \  c \in (\Z^d)^{q-k}, 
\end{align*}

\noi
acting on $u_{b_1}, \dots, u_{b_k}$
(where $k_n$ is as in \eqref{RTX0a}), 
we can view it as a linear operator
acting on the tensor product 
$u_{b_1}\otimes \cdots \otimes u_{b_k}$
to which 
we can apply Lemma \ref{LEM:RT}.
By noting\footnote{Note that 
a bound
$\|u_1\otimes \cdots \otimes u_k\|_{\bigotimes_{j = 1}^k B_j}
\le  \prod_{j = 1}^k \|u_{j}\|_{B_j}$
holds true for any tensor norm (such 
as the projective norm; see \cite[Section 2.1 and Chapter 6]{Ryan}).
}
\[
\|u_{b_1}\otimes \cdots \otimes u_{b_k}\|_{\bigotimes_{j = 1}^k \l^2_{b_j}}
= \prod_{j = 1}^k \|u_{k_j}\|_{\l^2_{b_j}}, 
\]

\noi
we then obtain a multilinear bound
for $H_{b_1, \dots, b_k, c}$.
We refer readers to 
\cite[Lemma 8.1]{BDNY}
on an improvement of this (trivial) argument
in the bilinear case.

\end{remark}

\subsection{Lattice point counting \& tensor norm estimates}

Given dyadic $N, N_1, N_2, N_3\ge 1$ and $m \in \Z$, 
define the base tensor $h^{{\bf N}, (m)}$ by 
\begin{align}
\label{baseT}
\begin{split}
h^{{\bf N}, (m)} &  = h^{N, N_1, N_2, N_3, (m)} (n,n_1,n_2,n_3) \\
& = \ind_{|n| \sim N}  \cdot 
\bigg( \prod_{j=1}^3 \ind_{|n_j|\sim N_j}  \bigg) \cdot \ind_{n \neq n_1,n_3} \cdot \ind_{\{n = n_1-n_2+n_3\}}  
\cdot  \ind_{\{\phi(\bar n) = m\}}, 
\end{split}
\end{align}

\noi
where $\phi(\bar n)$ is given by 
\begin{align}
\label{phi1}
\phi(\bar n) = |n|^2 - |n_1|^2 + |n_2|^2 - |n_3|^2.
\end{align}

\noi
Our main goal in this subsection is 
to establish basic bounds
on the base tensor $h^{{\bf N}, (m)}$;
see Lemma \ref{LEM:basetensor}.
In Section \ref{SEC:RT1} - \ref{SEC:RT4}, 
we estimate
the random tensor terms I - IV
by applying the random tensor estimate (Lemma \ref{LEM:RT})
and the tensor norm estimates (Lemma \ref{LEM:basetensor})
on the base tensor $h^{{\bf N}, (m)}$ (with weights).

For this purpose, 
we first establish some counting estimates.

\begin{lemma}
\label{LEM:count1}
{\rm (i)} 
%Let $\mathcal{R}=\mathbb{Z}$ or $\mathbb{Z}[i]$.
%Then, 
Given $m\in\Z\setminus \{0\}$
and $a_0,b_0\in \Z$, 
we have
\begin{equation}
\big|\big\{(a, b)\in \Z^2: m=ab,\,\,|a-a_0|\leq M,\,\,|b-b_0|\leq N\big\}\big|
= O(M^\eps N^\eps).
\label{count0}
\end{equation}

\noi
for any $\eps > 0$.

\medskip

\noi
{\rm (ii)}
Given 
dyadic $N, N_1, N_2, N_3\ge 1$ 
and  $m \in \Z$, define 
\begin{align}
\label{SN1}
\begin{split}
S^{{\bf N}, (m)} = \big\{(n, n_1, n_2, n_3)\in (\Z^2)^4
: \ & n=n_1-n_2+n_3, \
 n\ne n_1, n_3, \\
& \phi(\bar n) = m, \ 
|n|\sim N, 
|n_j|\sim  N_j, \ j = 1, 2, 3\big\}, 
\end{split}
\end{align}

\noi
where $\phi(\bar n)$ is as in \eqref{phi1}.
Given a subset $A$ of $\{n, n_1, n_2, n_3\}$, 
we then let $S^{{\bf N}, (m)}_A$ denote 
the collection of 
$(n, n_1,n_2,n_3)\in S^{{\bf N}, (m)}$
when the frequencies in $A$ are fixed.
For example, 
$S^{{\bf N}, (m)}_{n}$ denotes 
the collection of 
$(n, n_1,n_2,n_3)\in S^{{\bf N}, (m)}$
when the frequency $n$ is  fixed, 
while\footnote{For simplicity of notation, 
we used $n_1n_2$ to denote the set $\{n_1, n_2\}$.} 
$S^{{\bf N}, (m)}_{n_1n_2}$ denotes 
the collection of 
$(n, n_1,n_2,n_3)\in S^{{\bf N}, (m)}$
when the frequencies $n_1$ and $n_2$ are  fixed.
Then, the following bounds hold\textup{;}
\begin{align*}
\textup{(ii.a)}\ \ & 
|S^{{\bf N}, (m)}|  \les 
\min \big(N_1^2 N_3^2 (N\wedge N_2)^{\eps}, N^2 N_2^2 (N_1\wedge  N_3)^{\eps}, 
\\
& \hphantom{XXXXXXXl}
(N\wedge N_2)^2 (N_1N_3)^{1+\eps}, (N_1\wedge N_3)^2 (NN_2)^{1+\eps} \big), \\
\textup{(ii.b)}\ \ & 
|S_n^{{\bf N}, (m)}|\lesssim  \min \big(N_2^2 (N_1\wedge N_3)^{\eps}, 
(N_1N_2)^{1+\eps}, (N_1 N_3)^{1+\eps}, (N_2 N_3)^{1+\eps}\big), \\
\textup{(ii.c)}\ \ & 
|S^{{\bf N}, (m)}_{n_1}|\lesssim  \min \big(N_3^2 (N\wedge N_2)^{\eps}, (NN_2)^{1+\eps}, 
(N N_3)^{1+\eps}, 
(N_2 N_3)^{1+\eps} \big), \\
\textup{(ii.d)}\ \  &  
|S^{{\bf N}, (m)}_{n_2}|\lesssim  \min \big(N^2 (N_1\wedge N_3)^{\eps},
 (NN_1)^{1+\eps}, (N N_3)^{1+\eps}, 
 (N_1N_3)^{1+\eps}
\big), \\
\textup{(ii.e)}\ \ & 
|S^{{\bf N}, (m)}_{n_3}|\lesssim \min \big( N_1^2 (N\wedge N_2)^{\eps}, 
(N N_1)^{1+\eps}, 
(N N_2)^{1+\eps}, (N_1 N_2)^{1+\eps}\big), \\
\textup{(ii.f)}\ \ & 
|S^{{\bf N}, (m)}_{nn_1}|\lesssim  \min (N_2, N_3),\\
\textup{(ii.g)}\ \ & 
|S^{{\bf N}, (m)}_{nn_2}|\lesssim  \min (N_1, N_3)^{\eps}, \\
\textup{(ii.h)}\  \ & 
|S^{{\bf N}, (m)}_{nn_3}|\lesssim \min (N_1, N_2), \\
\textup{(ii.i)}\  \ & 
|S^{{\bf N}, (m)}_{n_1n_2}|\lesssim  \min(N, N_3), \\
\textup{(ii.j)}\ \ & 
 |S^{{\bf N}, (m)}_{n_1n_3}|\lesssim   \min (N, N_2)^{\eps},\\
 \textup{(ii.k)}\ \ & 
|S^{{\bf N}, (m)}_{n_2n_3}|\lesssim  \min (N, N_1), 
\end{align*}

\noi
where $a \wedge b  = \min (a,b)$.

\end{lemma}

\begin{proof} %[Proof of Lemma \ref{LEM:count}] 
(i) The bound \eqref{count0}
follows from a divisor counting argument; see \cite[Lemma~4.4]{DNY2} 
for details.

\smallskip

\noi
(ii)
Some of these counting estimates and similar forms have been previously studied;
see, for example, \cite[Lemmas 1 and 2]{BO96}, 
\cite[Section 4]{FOSW}, 
\cite[Lemma 4.4]{DNY2}, 
and 
\cite[Lemma 2.5]{DNY4}.\footnote{Lemma 2.4 in the arXiv version.} 
%%
%To illustrate the idea, we 
%estimate $|S^{{\bf N}, (m)}|$ in the following.

%part of (
We first note that the bound  (ii.f)
is trivial.
Given  $n$ and $n_1$, 
we fix $n_2$ which has $O(N_2)$ many choices.
This  in turn determines $n_3$,
yielding
$|S^{{\bf N}, (m)}_{nn_1}|\lesssim  N_2$.
By symmetry, we obtain 
$|S^{{\bf N}, (m)}_{nn_1}|\lesssim  N_3$.
The bounds (ii.h), (ii.i), and (ii.k) follow
from a similar consideration.

Next, we prove (ii.a).
We will  discuss four different ways to count the lattices in $S^{{\bf N}, (m)}$.
First, note that 
there are $O(N_1^2 N_3^2)$ many choices for $n_1$ and $n_3$.
With fixed $n_1$ and $n_3$,  
we count the number of choices for $n$ and $n_2$.
From \eqref{SN1} with \eqref{phi1}, 
it suffices to count $n$ satisfying
\begin{align}
|n| \sim N 
\quad \text{and}\quad 
|n|^2 +|n-n_1 -n_3|^2 & =c_1(n_1, n_3, m), 
\label{count1}
\end{align}

\noi
or $n_2$ satisfying
\begin{align}
|n_2|\sim N_2
\quad \text{and}\quad 
|n_2|^2 +|n_2-n_1-n_3|^2 & =c_1(n_1, n_3, m), 
\label{count1a}
\end{align}

\noi
where $c_1$ is a  constant depending on $n_1$, $n_3$, 
and $m$.
By rewriting the second condition in~\eqref{count1} as
\begin{align*}
 |2 n - n_1 -n_3 |^2 = c_2(n_1, n_3, m), 
\end{align*}

\noi
it suffices to count lattice points $n$ with $|n| \sim N$ on a given circle
(of radius $r > 0$), 
which is bounded by $O(N^\eps)$;  
see the proof of \cite[Lemma 1]{BO96}.
%where one divides the counting into two cases $\log N \ges \log r$
%or $\log N \ll \log r$
A similar argument shows that the number of choices of $n_2$
satisfying \eqref{count1a}
is $O(N_2^\eps)$.
Hence, we obtain
$|S^{{\bf N}, (m)}_{n_1n_3} |\lesssim  (N\wedge N_2)^{\eps}$,
which implies
$|S^{{\bf N}, (m)}|\lesssim N_1^2 N_3^2 (N\wedge N_2)^{\eps}$. 
By symmetry, we have
$|S^{{\bf N}, (m)}_{nn_2} |\lesssim  (N_1\wedge N_3)^{\eps}$,
which implies
$|S^{{\bf N}, (m)}|\les 
N^2 N_2^2 (N_1\wedge  N_3)^{\eps}$.

Let us proceed differently by fixing
 $n_2$ first.
Then,
it suffices to count
$(n_1, n_3)$ satisfying
\begin{align*}
2(n_2-n_1)\cdot (n_2-n_3) = m.
\end{align*}

\noi
We denote the first (and second) coordinates
of $n_j$ by $n_j^{(1)}$  (and $n_j^{(2)}$, respectively).
Let
\[
(n_2^{(i)}- n_1^{(i)}) (n_2^{(i)} - n_3^{(i)}) = c_i
\]
for $i = 1,2$.
We first consider the case $c_1 = c_2 = 0$.
Due to the condition $n_2 \neq n_1, n_3$, without loss of generality, we may assume $n_2^{(1)} = n_1^{(1)}$ and $n_2^{(2)} = n_3^{(2)}$. Namely,  $n_1^{(1)}$ and $n_3^{(2)}$ are fixed 
since $n_2$ has been fixed.
Then,  there are $O(N_1N_3)$ many choices 
for $n_1^{(2)}$ and $n_3^{(1)}$,
(and thus for  $n_1$ and $n_3$), 
 which implies $|S^{{\bf N}, (m)}| \les N_2^2 (N_1N_3)$ in this case.
In the following, we assume that $c_1 \neq 0$ or $c_2 \neq 0$.
By symmetry, we only consider the case $c_2 \neq 0$.
By fixing the first coordinates $n_1^{(1)}$ and $n_3^{(1)}$ 
(which have $O(N_1N_3)$ many choices),
we count the number of choices for the second coordinates
$n_1^{(2)}$ and $n_3^{(2)}$,  satisfying
\[
(n_2^{(2)}- n_1^{(2)}) (n_2^{(2)} - n_3^{(2)})
= c_2\]

\noi
for some fixed $c_2 \in \R \setminus\{0\}$.
It follows from \eqref{count0}
that
 $|S^{{\bf N}, (m)}_{n_2}| \les  (N_1N_3)^{1+\eps}$,
which implies $|S^{{\bf N}, (m)}| \les N_2^2 (N_1N_3)^{1+\eps}$.
By symmetry, we obtain
 $|S^{{\bf N}, (m)}_{n}| \les  (N_1N_3)^{1+\eps}$,
which implies $|S^{{\bf N}, (m)}| \les N^2 (N_1N_3)^{1+\eps}$. 
Also, 
 $|S^{{\bf N}, (m)}_{n_j}| \les  (NN_2)^{1+\eps}$, $j = 1, 3$, 
 and 
$|S^{{\bf N}, (m)}| \les (N_1\wedge N_3)^2 (NN_2)^{1+\eps}$.

Lastly, we prove (ii.b).
The bounds (ii.c), (ii.d), and (ii.e) follow from symmetry.
We have already shown  $|S^{{\bf N}, (m)}_{n}| \les  (N_1N_3)^{1+\eps}$.
Given $n$, 
we fix  $n_2$ which has $O(N_2^2)$ many choices.
Then, %from \eqref{count0}, 
by repeating the first part of the proof of (ii.a), 
we see that there are $O((N_1\wedge N_3)^\eps)$ many choices
for $n_1$ and $n_3$, 
which yields
$|S^{{\bf N}, (m)}_{n}|\les N_2^2 (N_1\wedge N_3)^{\eps}$.
It remains to prove
$|S^{{\bf N}, (m)}_{n}| \les(N_1N_2)^{1+\eps}$, since $|S^{{\bf N}, (m)}_{n}| \les (N_2 N_3)^{1+\eps}$
follows from symmetry.
Given $n$, 
it suffices to count
$(n_1, n_2)$ satisfying
\begin{align}
2(n_2-n_1)\cdot (n_1-n) = m.
\label{count3}
\end{align}

\noi
By writing \eqref{count3}
as 
\[
(n_2^{(i)}- n_1^{(i)}) (n_1^{(i)} - n^{(i)}) = c_i, \quad i = 1, 2, 
\]

\noi
we first consider the case $c_1 = c_2 = 0$.
Due to the condition $n_1 \neq n, n_2$, without loss of generality, we may assume 
$n_1^{(1)} = n^{(1)}$ and $n_1^{(2)} = n_2^{(2)}$. 
Given $n$, 
there are $O(N_1)$ many choices for $n_1^{(2)}$
(which fixes $n_2^{(2)}$)
and $O(N_2)$ many choices for $n_2^{(1)}$.

Next, we consider the case  $c_1 \neq 0$ or $c_2 \neq 0$.
By symmetry, we only consider the case $c_2 \neq 0$.
By fixing the first coordinates $n_1^{(1)}$ and $n_2^{(1)}$ 
(which have $O(N_1N_2)$ many choices),
we count the number of choices for the second coordinates
$n_1^{(2)}$ and $n_2^{(2)}$,  satisfying
\[
(n_2^{(2)}- n_1^{(2)}) (n_1^{(2)} - n^{(2)})
= c_2\]

\noi
for some fixed $c_2 \in \R \setminus\{0\}$.
It follows from \eqref{count0}
(with $a = n_2^{(2)}- n_1^{(2)}$, $a_0 = 0$, 
$b = n_1^{(2)} - n^{(2)}$, and $b_0 = -  n^{(2)}$) 
that
 $|S^{{\bf N}, (m)}_{n}| \les  (N_1N_2)^{1+\eps}$.
\end{proof}

\begin{remark}\label{REM:count}
\rm
Note that the condition $n \neq n_1, n_3$ in \eqref{SN1} is crucial
in the cases where we used the divisor counting argument.
Define
a set $\wt S^{{\bf N}, (m)}$
by \eqref{SN1} after dropping the condition 
 $n\ne n_1, n_3$.
 Let $m =0$ and consider
\begin{align*}
\wt S^{{\bf N}, (0)}
\cap \{n = n_1\}
& = \{ (n_1,n_2) \in (\Z^2)^2: |n_j| \sim N_j, \ j = 1, 2\}.
\end{align*}

\noi
and thus $|\wt S^{{\bf N}, (0)}| \ges N_1^2 N_2^2$.
When $N \sim N_1 \gg N_2 \sim N_3$, 
the  bound above is much worse than (ii.a) in Lemma \ref{LEM:count1}.
\end{remark}

The next lemma follows from Schur's test and the counting lemma (Lemma \ref{LEM:count1}).

\begin{lemma}
\label{LEM:basetensor}

Given $\eps > 0$, the following estimates hold\textup{;}
\begin{align}
\|h^{{\bf N},(m)} \|_{n_1 n_2 n_3 \to n} &\lesssim  \min \big(N_2 (N_1\wedge N_3)^{\eps}, (N_{\med}N_{\min})^{\frac12 +\eps}\big), 
\notag \\
\|h^{{\bf N},(m)} \|_{n_1 \to n n_2n_3  } 
&\lesssim  \min \big(N_3 (N\wedge N_2)^{\eps}, (NN_2)^{\frac12+\eps}, (N N_3)^{\frac12+\eps}, (N_2 N_3)^{\frac12+\eps}\big), 
\notag \\
\|h^{{\bf N},(m)} \|_{ n_2 \to n n_1n_3 } &\lesssim  \min \big(N (N_1\wedge N_3)^{\eps}, 
(N N_1)^{\frac12+\eps}, (N N_3)^{\frac12+\eps}, 
(N_1N_3)^{\frac12+\eps}\big), 
\notag \\
\|h^{{\bf N},(m)} \|_{n_3 \to n n_1n_2} &\lesssim \min 
\big( N_1 (N\wedge N_2)^{\eps}, (N N_1)^{\frac12+\eps},
(N N_2)^{\frac12+\eps},
 (N_1 N_2)^{\frac12+\eps}\big), 
\label{base1}\\
\|h^{{\bf N},(m)} \|_{ n_2n_3\to nn_1} 
& \lesssim    \min (N, N_1)^{\frac12}
\cdot \min (N_2, N_3)^{\frac12}, 
\notag \\
\|h^{{\bf N},(m)} \|_{n_1n_3 \to nn_2} 
&\lesssim  
 \min (N, N_2)^{\eps}\cdot 
\min (N_1, N_3)^{\eps},  
\notag \\
\|h^{{\bf N},(m)} \|_{n_1n_2\to nn_3} 
&\lesssim 
 \min \big(N, N_3\big)^{\frac 12 }\cdot 
\min (N_1, N_2)^{\frac12}, 
\notag
\end{align}

\noi
uniformly in 
dyadic $N, N_1, N_2, N_3\ge 1$ and $m \in \Z$, 
where 
%$N_{\max}$, 
$N_{\med}$ and $N_{\min}$ are as in 
 \eqref{ord1}.

\end{lemma}

\begin{proof}
By Schur's test with \eqref{baseT}, 
\begin{align*}
\| h^{{\bf N}, (m)} \|_{n_1n_2n_3 \to n}
& \le \bigg( \sup_{n \in \Z^2} \sum_{n_1, n_2, n_3 \in \Z^2}
h^{{\bf N}, (m)} \bigg)^\frac 12 
\bigg(\sup_{n_1, n_2, n_3 \in \Z^2} \sum_{n \in \Z^2}
h^{{\bf N}, (m)}\bigg)^\frac 12 \\
& =   |S_n^{{\bf N}, (m)}|^\frac 12 |S^{{\bf N}, (m)}_{n_1, n_2, n_3}|^\frac 12 \\
& \le   |S_n^{{\bf N}, (m)}|^\frac 12, 
\end{align*}

\noi
where the third step follows
from the fact that $n$ is uniquely determined 
by $n_1$, $n_2$, and $n_3$ under $n = n_1 - n_2 + n_3$.
Then, the first bound in \eqref{base1}
follows from Lemma \ref{LEM:count1}.

The other bounds in \eqref{base1}
follow from 
Lemma \ref{LEM:count1}, 
once we note
\begin{align*}
%    \big\| h^{(m)}_{nn_1n_2n_3}  \big\|_{nn_1n_2n_3} & \le |S|^{\frac12}\\
\| h^{{\bf N},(m)}  \|_{n_1\to nn_2n_3} & \le |S^{{\bf N},(m)}_{n_1}|^{\frac12},  \\
\| h^{{\bf N},(m)}  \|_{n_2\to nn_1n_3 } & \le |S^{{\bf N},(m)}_{n_2}|^{\frac12},  \\
\| h^{{\bf N},(m)}  \|_{n_3\to nn_1n_2 } & \le |S^{{\bf N},(m)}_{n_3}|^{\frac12},  \\
\| h^{{\bf N},(m)}  \|_{n_2n_3\to nn_1} & \le 
|S^{{\bf N},(m)}_{n_2n_3}|^{\frac12} 
|S^{{\bf N},(m)}_{nn_1}|^{\frac12} , 
\\
\| h^{{\bf N},(m)}  \|_{n_1n_3 \to nn_2} & \le |S^{{\bf N},(m)}_{n_1n_3}|^{\frac12}
 |S^{{\bf N},(m)}_{nn_2}|^{\frac12},   \\
\| h^{{\bf N},(m)}  \|_{n_1n_2 \to n n_3} & \le |S^{{\bf N},(m)}_{n_1n_2}|^{\frac12} 
|S^{{\bf N},(m)}_{nn_3}|^{\frac12}, 
\end{align*}

\noi
each of 
which follows from Schur's test as above.
\end{proof}

\section{Resonant contribution}\label{SEC:res}

In this section, we estimate the
resonant term $\RR(w+z)$
by  following 
the presentation in 
 \cite[Subsection 4.2]{CO}.
Due to the multilinearity and the symmetry, it suffices to treat
\[\RR(w, w, w), \quad 
\RR(z, z, z), \quad
\RR(w, z, z), \quad \text{and}\quad 
\RR(w, w, z).\]

\noi
Since 
the complex conjugate on the second argument of $\RR(v_1, v_2, v_3)$
in \eqref{non1}  
does not play any role, we drop the complex conjugate sign 
in the following.

\medskip

\noi 
$\bullet$ {\bf Case 1:} $\RR(w, w, w)$.
\\ \indent
This case is already included in 
\eqref{tri1}, giving
\begin{align*}
\|\RR(w, w, w)\|_{X^{s, -\frac 12 +2 \eps}_T} 
\les 1
%\label{R1}
\end{align*}

\noi
for any $w$ with $\|w\|_{X^{s, \frac 12 + \eps}_T} \le 1$
and $0 < T \le 1$.

\noi $\bullet$ {\bf Case 2:} 
$\RR(z, z, z)$.
\\ \indent
 By H\"older's inequality  ($\frac{1}{2} =\frac{1}{p} + \frac{1}{q}$
 with $p \gg1$ and $q > 2$ close to $2$ such that $(\frac 12 - 2\eps ) q > 1$), 
 \eqref{lin1}, 
Young's inequality, \eqref{decay}, 
and  Lemma \ref{LEM:prob1}, we have 
\begin{align*}
& \|\RR(z, z, z)\|_{X^{s, -\frac 12 +2 \eps}_T} \\
 & \quad \les \sup_{n \in \Z^2} \|\jb{\tau + |n|^2}^{-\frac{1}{2}+2\eps}\|_{L^{q}_\tau}
\bigg\| \jb{n}^{s-3} |g_n|^3 \intt_{\tau = \tau_1 - \tau_2 + \tau_3} 
\prod_{j = 1}^3 \ft \eta_{_T}(\tau_j + |n|^2) d\tau_1 d\tau_2 \bigg\|_{\l^2_{n} L^p_\tau}\\
& \quad \les
T^{1-\frac 1p} \big\|\jb{n}^{s-3} |g_n |^3 \big\|_{\l^2_n} 
\les T^{1 - \frac 1p-3\dl_1 } \|\jb{n}^{s-3+ 3\dl_2} \|_{\l^2_n} 
\les 1,   
\end{align*}

\noi 
%for some small $\ta > 0$, 
yielding \eqref{tri2}, 
outside an exceptional set of probability $< e^{-\frac{c_1}{T^{c_2}}}$, 
provided that $\dl_1, \dl_2 > 0$
are sufficiently small
and $s < 2 - 3\dl_2$.

\medskip

\noi 
$\bullet$ {\bf Case 3:} 
$\RR(w, z, z)$.
\\ \indent
By H\"older's inequality (with $p$ and $q$ as in Case 2),
 \eqref{lin1}, 
 Young's inequality, \eqref{decay}, and 
Lemma~\ref{LEM:prob1}, we have 
\begin{align}
\begin{split}
& \|\RR(w, z, z)\|_{X^{s, -\frac 12 +2 \eps}_T} \\
& \quad  \les 
\bigg\| \jb{n}^{s-2} |g_n|^2 
\intt_{\tau = \tau_1 - \tau_2 + \tau_3} 
\ft {\wt w}(n, \tau_1)
\prod_{j = 2}^3 \ft \eta_{_T}(\tau_j + |n|^2) d\tau_1 d\tau_2 \bigg\|_{\l^2_{n} L^p_\tau}\\
& \quad \les
T^{\frac{1}{2}-\frac 1p} \Big(\sup_{n\in \Z^2} \jb{n}^{-2} |g_n|^2\Big) \|\jb{n}^{s} \ft{\wt w} (n, \tau) \|_{\l^2_n L^2_\tau} 
\les T^{\frac 12 - \frac 1p - 2\dl_1} \|\wt w \|_{X^{s, 0}}
\end{split}
\label{R2}
\end{align}

\noi 
for any extension   $\wt w$  of $w$ with  $\|w \|_{X^{s, \frac{1}{2}+\eps}_T} \leq 1$, 
outside an exceptional set of probability $< e^{-\frac{c_1}{T^{c_2}}}$, 
provided that $\dl_1, \dl_2 > 0$
are sufficiently small.
Hence, by taking an infimum over all extensions $\wt w$ of $w$, 
it follows from \eqref{R2} and \eqref{Xsb2} that 
\begin{align*}
 \|\RR(w, z, z)\|_{X^{s, -\frac 12 +2 \eps}_T} 
\les  \| w \|_{X^{s, \frac 12 + \eps}_T}
\le 1
\end{align*}

\noi
outside an exceptional set of probability $< e^{-\frac{c_1}{T^{c_2}}}$, 
yielding 
\eqref{tri2} in this case.

\medskip

\noi $\bullet$ {\bf Case 4:}
$\RR(w, w, z)$.
\\ \indent
By H\"older's inequality (with $p$ and $q$ as in Case 2),
 \eqref{lin1}, 
 Young's inequality with \eqref{decay},  Lemma \ref{LEM:prob1}, 
H\"older's inequality (in $\tau$), 
and  $\l^2_n \subset \l^4_n$,
we have 
\begin{align}
\begin{split}
&  \|\RR(w, w, z)\|_{X^{s, -\frac 12 +2 \eps}_T}  \\
 &\quad  \les
\bigg\| \jb{n}^{s-1} |g_n| \intt_{\tau = \tau_1 - \tau_2 + \tau_3}
\prod_{j = 1}^2 \ft {\wt w}(n, \tau_j) 
\cdot  \ft \eta_{_T}(\tau_3 + |n|^2)d\tau_1 d\tau_2 \bigg\|_{\l^2_{n} L^p_\tau}\\
& \quad \les
T^{\frac{1}{2}-\frac 1p} 
\Big(\sup_{n \in \Z^2} \jb{n}^{-\dl_2} |g_n|\Big)
\| \jb{n}^s \ft{\wt{w}} (n, \tau)\|_{\l^4_n L^\frac{4}{3}_\tau}^2 \\
& \quad
\les T^{\frac{1}{2}-\frac 1p - \dl_1}
\|\wt w \|_{X^{s, \frac{1}{4}+\eps}}^2
\end{split}
\label{R3}
\end{align}

\noi 
for any extension   $\wt w$  of $w$ with  $\|w \|_{X^{s, \frac{1}{2}+\eps}_T} \leq 1$, 
outside an exceptional set of probability $< e^{-\frac{c_1}{T^{c_2}}}$, 
provided that $\dl_1, \dl_2 > 0$
are sufficiently small
and $s \ge -1+\dl_2$.
Hence, by taking an infimum over all extensions $\wt w$ of $w$, 
it follows from \eqref{R3} and \eqref{Xsb2} that 
\begin{align*}
 \|\RR(w, w, z)\|_{X^{s, -\frac 12 +2 \eps}_T} 
\les  \| w \|_{X^{s, \frac 12 + \eps}_T}^2
\le 1
\end{align*}

\noi
outside an exceptional set of probability $< e^{-\frac{c_1}{T^{c_2}}}$, 
yielding 
\eqref{tri2} in this case.

Therefore, the conclusion of Proposition \ref{PROP:1}
holds
for the resonant term $\RR(w+z)$.

%\begin{remark} \label{REM:local}\rm
%In this section, 
%we carefully used the definition \eqref{Xsb2} of the local-in-time space $X^{s, b}_T$.
%Strictly speaking, such a care must be taken in all the subsequent  analysis.
%However, this is a routine work and, for simplicity of presentation, we write estimates 
%directly with 
%$\|w\|_{X^{s, b}_T}$ in the following, 
%meaning that the same estimates hold with $\|\wt w\|_{X^{s, b}}$ 
%for any extension $\wt w$ of $w$, 
%allowing us to take an infimum over $\wt w$.
%
%\end{remark}
%

\section{Purely stochastic term}
\label{SEC:PS1}

In this section, we study the purely stochastic term $\NN(z, z, z)$,
where $\NN$ and $z$ are
as in~\eqref{non1}
and \eqref{lin1}, respectively.

\begin{proposition}
\label{PROP:PS1}

Let $s < \frac 12$.
Given small $\eps > 0$, 
 there exists $\ta > 0$
such that 
\begin{align}
\label{PS1}
\Big\|
\| \NN(z, z, z)\|_{X^{s,-b'}_T}\Big\|_{L^p(\O)}
\les p^\frac 32 T^\ta
\end{align}

\noi
for any finite $p \ge 1$ and $0 < T\ll 1$, where 
 $b' = \frac 12 - 2\eps$
is  as in \eqref{b1}.
In particular, 
 the conclusion of Proposition \ref{PROP:1}
holds
for the purely stochastic term $\NN(z, z, z)$.

\end{proposition}

\begin{proof} %[Proof of Proposition \ref{PROP:PS1}]

We first prove 
\begin{align}
\label{PS4}
\begin{split}
\Big\|\| \Q_{N} 
\NN(\Q_{N_1} z, \Q_{N_2} z, \Q_{N_3} z) \|_{X^{s,-b'}_T}\Big\|_{L^p(\O)}
 \les p^\frac 32 T^\ta \frac{N_{\max}^{ \dl} N^{s}} {N_1N_2N_3}
\sup_{m \in \Z} |S^{{\bf N}, (m)}|^\frac 12 
\end{split}
\end{align}

\noi 
for small $\ta> 0$ and any small $ \dl > 0$, uniformly in dyadic $N, N_1, N_2, N_3 \ge 1$, 
where $\Q_N$ is the Littlewood-Paley
 projector
and 
$S^{{\bf N}, (m)}$ is as in \eqref{SN1}.

From \eqref{non1}
and \eqref{lin1}, we have 
\begin{align*}
%\begin{split}
&  \F_{x, t}\big(e^{-it \Delta}  
\eta_{_T}
\Q_N 
\NN(\Q_{N_1} z, \Q_{N_2} z, \Q_{N_3} z) \big)(n, \tau)\\
& \quad = 
%\sum_{\substack{n \in \Z^2\\|n|\sim N}}
%e^{in \cdot x} 
\ind_{|n|\sim N}
\sum_{\substack{n = n_1 - n_2 + n_3\\ n \neq n_1, n_3}} 
\ft \eta_{_T} (\tau - \phi (\bar n))
\bigg( \prod_{j=1}^3 \frac{g_{n_j}^*}{\jb{n_j}}\ind_{|n_j|\sim N_j}  \bigg), 
%\end{split}
%\label{PS5}
\end{align*}

\noi 
where $\phi(\bar n)$ is as in \eqref{phi1}
and $g^*_{n_j}$ is defined by 
\begin{align}
g^*_{n_j} = 
\begin{cases}
g_{n_j}, & \text{if } j = 1, 3,\\
\cj{g_{n_2}}, & \text{if } j = 2.
\end{cases}
\label{g1}
\end{align}

%
%\noi
%Them, 
%from \eqref{Xsb2}, we have 
%\begin{align*}
%& \|\Q_{N} 
%\NN(\Q_{N_1} z, \Q_{N_2} z, \Q_{N_3} z)
%\|_{X^{s,-b'}_T }^2 \\
%& \quad \le \|\eta_{_T}\Q_{N} 
%\NN(\Q_{N_1} z, \Q_{N_2} z, \Q_{N_3} z)
%\|_{X^{s,-b'} }^2 \\
%& \quad \le  \sum_{m \in \Z} \sum_{\substack{n \in \Z^2\\|n|\sim N}} 
%\jb{m}^{- 2b'} \jb{n}^{2s} 
%\Bigg|  \sum_{\substack{n = n_1 - n_2 + n_3\\ n \neq n_1, n_3 \\ \phi (\bar n) = m}} 
%\bigg( \prod_{j=1}^3 \frac{g_{n_j}^*}{\jb{n_j}}\ind_{|n_j|\sim N_j}  \bigg)
%\Bigg|
%\end{align*}
%
%
%\noi 
%for any $0 < T\le 1$.

\noi
Then,  from 
Minkowski's integral inequality, the Wiener chaos estimate (Lemma~\ref{LEM:hyp}), 
and
the orthogonality of 
$\big\{\ind_{n_2\ne n_1, n_3} \prod_{j=1}^3 g_{n_j}^*\big\}_{n_1, n_2, n_3\in \Z^2}$
in $L^2(\O)$
(as elements in the homogeneous Wiener chaos of order $3$), 
we have 
\begin{align}
\begin{split}
& \text{LHS of \eqref{PS4}}\\
& \les p^\frac 32 
\Bigg\|% \sum_{m \in \Z} \sum_{\substack{n \in \Z^2\\|n|\sim N}}
 \jb{\tau}^{-b'} \jb{n}^{s}  
 \bigg\|\sum_{\substack{n = n_1 - n_2 + n_3\\ n \neq n_1, n_3}} % \\ \phi (\bar n) = m}} 
\ft \eta_{_T} (\tau - \phi (\bar n))
\bigg( \prod_{j=1}^3 \frac{g_{n_j}^*}{\jb{n_j}}\ind_{|n_j|\sim N_j}  \bigg)
\bigg\|_{L^2(\O)}
\Bigg\|_{L^2_\tau \l^2_{|n|\sim N}}
\\
& \les p^\frac 32 
\bigg\|
\jb{\tau }^{-b'} \jb{n}^{s}  
\ind_{A(n)} %\ind_{\phi (\bar n) = m}
\ft \eta_{_T} (\tau - \phi (\bar n))
\prod_{j=1}^3 \frac{1}{\jb{n_j}}  
\bigg\|_{L^2_\tau  \l^2_{|n|\sim N} \l^2_{n_1, n_2}}, 
\end{split}
\label{PS7}
\end{align}

\noi
where $A(n)$ is given by 
\begin{align}
A(n) = \big\{(n_1, n_2, n_3) \in (\Z^2)^3:
n = n_1 - n_2 + n_3, \ 
n \ne n_1, n_3, \ |n_j|\sim N_j\big\}.
 \label{RT12a}
\end{align}

\noi
From 
Lemma \ref{LEM:conv} with \eqref{b1}, we have 
\begin{align}
\begin{split}
 \int_\R  \jb{\tau }^{-2b'}|\ft \eta_{_T} (\tau - m)|^2d\tau 
&  \les T^{2} 
\sum_{\mu \in \Z}
\int_{\mu - \frac 12}^{\mu+\frac 12} 
\frac{1}{\jb{\tau }^{2b'}\jb{T(\tau - m)}^{1+\eps}}d\tau \\
& \les T^{1-\eps} 
\sum_{\mu \in \Z}
\frac{1}{\jb{\mu }^{2b'}\jb{\mu - m}^{1+\eps}}\\
& \les T^{1-\eps} 
\jb{m}^{-2b'}
\end{split}
\label{PS8}
\end{align}

\noi
for any $0 < T \le 1$, uniformly in $m \in \Z$.
Thus, from \eqref{PS7} and \eqref{PS8}, we obtain
\begin{align}
 \text{LHS of \eqref{PS4}}
 \les p^\frac 32 T^\ta
\bigg\|
\jb{m}^{-b'} \jb{n}^{s}  
\ind_{A(n)} \ind_{\phi (\bar n) = m}
\prod_{j=1}^3 \frac{1}{\jb{n_j}}  
\bigg\|_{\l^2_m  (\Z) \l^2_{|n|\sim N} \l^2_{n_1, n_2}}
\label{PS9}
\end{align}

\noi
for some $\ta > 0$.
Recall from 
\eqref{b1}
that $-2b' = -1 + 4\eps$.
Then, by noting 
from \eqref{phi1}
that the summand in~\eqref{PS9} is non-trivial
only for  $|m | \les N_{\max}^2$
(where $N_{\max}$ is as in \eqref{ord1}), we have
 \begin{align*}
 \text{LHS of \eqref{PS4}}
& \les p^\frac 32 T^\ta
 \frac{ N_{\max}^{4\eps}N^{s}}{N_1 N_2 N_3}  \sup_{m \in \Z} 
\Bigg( \sum_{\substack{n, n_1, n_2, n_3 \in \Z^2\\n =  n_1 - n_2 + n_3\\ n \neq n_1, n_3 \\ \phi (\bar n) = m}}  
\ind_{|n|\sim N} \prod_{j=1}^3 \ind_{|n_j|\sim N_j} \Bigg)^\frac 12 \\
& = 
p^\frac 32 T^\ta
 \frac{ N_{\max}^{4\eps}N^{s}}{N_1 N_2 N_3} \sup_{m \in \Z} 
|S^{{\bf N}, (m)}|^\frac 12.
\end{align*}

\noi 
This proves  \eqref{PS4}.

%\[
%N_{*} = \max (N_0,N_1,N_2,N_3).
%\]

From (ii.a)  in Lemma \ref{LEM:count1}, we have  
\begin{align}
|S^{{\bf N}, (m)}|^\frac 12
\les N_{\max}^{\frac 12 + \eps}
N_{\med}^{\frac 12 + \eps}
N_{\min}, 
\label{PS10}
\end{align}

\noi
where
$N_{\max}$, 
$N_{\med}$, and 
$N_{\min}$ are as in \eqref{ord1}.
Then, 
by summing 
  \eqref{PS4} over dyadic blocks $N, N_1, N_2, N_3 \ge1$
and applying \eqref{PS10}, 
we have 
\begin{align*}
\Big\|
 \| \NN(z, z, z)\|_{X^{s,-b'}_T}\Big\|_{L^p(\O)}
&  \les 
p^\frac 32 T^\ta
\sup_{\substack{N, N_1, N_2, N_3\ge 1\\\text{dyadic}}}
 \frac{ N_{\max}^{5\eps }N^{s}}{N_1 N_2 N_3} \sup_{m \in \Z} 
|S^{{\bf N}, (m)}|^\frac 12\\
& \les 
p^\frac 32 T^\ta
\sup_{\substack{N, N_1, N_2, N_3\ge 1\\\text{dyadic}}}
N_{\max}^{-\frac 12 + s +6\eps}
N_{\med}^{-\frac 12 + \eps}\\
& \les 
p^\frac 32 T^\ta, 
\end{align*}

\noi
provided that $s < \frac 12$ and $\eps > 0$ is sufficiently small.
This proves \eqref{PS1}.
The second claim in 
Proposition \ref{PROP:PS1}
follows  from 
\eqref{PS1} and 
Chebyshev's inequality (see also Lemma 4.5 in~\cite{TzvBO}).
\end{proof}

\section{Random tensor term I}
\label{SEC:RT1}

In this section, we estimate
the  random tensor term
 $\NN(z, w, z)$. 
Define the random operator $\TT_1$
by setting
\begin{align}\label{RT0}
\TT_1(w) =  \NN (z, w,z), 
\end{align}

\noi
where $z$ denotes the random linear solution in \eqref{lin1}.
Then, our main goal is to prove the following proposition.

\begin{proposition}
\label{PROP:RT1}
Given small   $s > 0$ and $\eps > 0$, 
 there exists $\ta > 0$
such that 
\begin{align}
\Big\| \| \TT_1 \|_{X^{s,b}_T \to 
X^{s,-b'}_T} \Big\|_{L^p (\O)} 
\les p  T^\ta
\label{RT0a}
\end{align}

\noi
for any finite $p \ge 1$ and $0 < T\ll 1$, where 
$b = \frac 12 + \eps$ and $b' = \frac 12 - 2\eps$
are as in \eqref{b1}.
In particular, 
 the conclusion of Proposition \ref{PROP:1}
holds
for the random tensor term $\NN(z, w, z)$.

\end{proposition}

\begin{proof} %[Proof of Proposition \ref{PROP:RT1}]

We write
\begin{align*}
\TT_1(w) = \sum_{\substack{N_1,N_3\ge 1\\\text{dyadic}}} \TT_1^{N_1,N_3} (w),
%\label{RT1}
\end{align*}

\noi
where $\TT_1^{N_1,N_3}$ is given by 
\begin{align}
\TT_1^{N_1,N_3} (w) =  \NN(\Q_{N_1} z, w,\Q_{N_3} z).
\label{RT1a}
\end{align}

\noi
Here, $\Q_{N_j}$ denotes the Littlewood-Paley projector.
In the following, we prove
that there exist small $\ta, \dl > 0$ such that 
\begin{align}
\Big\| \| \TT_1^{N_1,N_3} \|_{X^{s,b}_T \to 
X^{s,-b'}_T} \Big\|_{L^p (\O)} 
\les p  T^\ta (N_1N_3)^{-\dl} 
\label{RT2}
\end{align}

\noi
for any finite $p \ge 1$ and $0 < T\ll 1$.
Then, 
\eqref{RT0a} follows from summing over dyadic
$N_1, N_3 \ge 1$, 
and the second claim in 
Proposition \ref{PROP:RT1}
follows  from Chebyshev's inequality. % (see also Lemma 4.5 in~\cite{TzvBO}).

From \eqref{RT1a} with \eqref{RT0} and \eqref{non1}, we have 
\begin{align*}
 \F_{x, t} &
\big(\eta_{_T}\TT_1^{N_1,N_3} (w)\big)
 (n,\tau) \\
& = \sum_{\substack{n = n_1 - n_2 + n_3\\ n \neq n_1,n_3}} 
\bigg(\prod_{j \in \{1, 3\}}
\frac{g_{n_j} }{\jb{n_j}}
\ind_{|n_j|\sim N_j}\bigg)\\
& \quad \times  \int_{\R} \ft \eta_{_T} (\tau + \tau_2 + |n_1|^2 - |n_2|^2 + |n_3|^2) 
\cj{\ft w (n_2, \tau_2 - |n_2|^2)} d\tau_2.
\end{align*} 

\noi 
Thus, we have 
\begin{align}
\begin{split}
 \F_{x, t} &
\big(\eta_{_T} \TT_1^{N_1,N_3} (w)\big)
 (n,\tau -|n|^2) \\
&= \jb{n}^{-s} \sum_{n_2\in \Z^2 } \jb{n_2}^s 
\int_\R 
\Hb_1^{N_1, N_3} (n,n_2,\tau, \tau_2) 
\cj{ \ft w(n_2, \tau_2 - |n_2|^2) } 
d\tau_2,
\end{split}
\label{RT3}
\end{align} 

\noi
where 
\begin{align}
\begin{split}
& \Hb_1^{N_1, N_3} (n,n_2,\tau, \tau_2)  \\
& \quad =  \frac{\jb{n}^{s}}{ \jb{n_2}^{s}} 
\sum_{\substack{n + n_2 = n_1 + n_3\\ n \neq n_1,n_3}} 
\bigg(\prod_{j \in \{1, 3\}}
\frac{g_{n_j} }{\jb{n_j}}
\ind_{|n_j|\sim N_j}\bigg)
\ft \eta_{_T} (\tau + \tau_2 - \phi (\bar n))  
\end{split}
\label{RT4}
\end{align}

\noi
with  $\phi (\bar n)$ as in \eqref{phi1}. 
Then, 
from 
\eqref{Xsb2}, \eqref{RT3}, 
and Cauchy-Schwarz's inequality, 
we have 
\begin{align}
\begin{split}
 \| &  \TT_1^{N_1,N_3} (w) \|_{X^{s,  -b'}_T} 
\le \| \eta_{_T} \TT_1^{N_1,N_3} (\wt w)
\|_{X^{s,  -b'}} \\
& = \big\| \jb{n}^s \jb{\tau}^{-b'}
 \F_{x, t} 
\big(\eta_{_T} \TT_1^{N_1,N_3} (\wt w)\big)
 (n,\tau -|n|^2) 
 \big \|_{\l^2_n L_\tau^2} \\ 
& = \bigg\| \jb{\tau}^{- b'}
\sum_{n_2\in \Z^2} \jb{n_2}^s 
\int_\R 
\Hb_1^{N_1, N_3} (n,n_2,\tau, \tau_2) 
\cj{ \ft {\wt w}(n_2, \tau_2 - |n_2|^2) } 
d\tau_2
\bigg\|_{ \l^2_n L^2_{\tau}}\\  
& \le \bigg\|\jb{\tau}^{- b'}  \int_\R  
 \| \Hb_1^{N_1, N_3} (n,n_2,\tau, \tau_2) \|_{\l^2_{n_2} \to \l^2_n} 
 \| \jb{n_2}^s  \ft {\wt w} (n_2,\tau_2 - |n_2|^2) \|_{\l^2_{n_2}}   d\tau_2 \bigg\|_{L^2_{\tau}} \\
& \le \big\| \jb{\tau}^{- b'} \jb{\tau_2}^{-b} 
\| \Hb_1^{N_1, N_3}  (n,n_2,\tau, \tau_2) \|_{\l^2_{n_2} \to \l^2_n} \big\|_{L^2_{\tau, \tau_2}} \cdot \|\wt w\|_{X^{s,b}}
\end{split} 
\label{RT5}
\end{align}

\noi 
for any extension $\wt w$ of $w$.
By taking an infimum over all extensions $\wt w$ of $w$
and then a supremum over $w$ with $\|w\|_{ X^{s,b}_T} \le 1$, 
we  obtain 
\begin{align}
 \|  \TT_1^{N_1,N_3}  \|_{X^{s,  b}_T\to  X^{s,-b'}_T} 
 \le \big\| \jb{\tau}^{- b'} \jb{\tau_2}^{-b} 
\| \Hb_1^{N_1, N_3}  (n,n_2,\tau, \tau_2) \|_{\l^2_{n_2} \to \l^2_n} \big\|_{L^2_{\tau, \tau_2}}.
\label{RT6}
\end{align}

\noi
Hence, \eqref{RT2} follows once we prove
\begin{align}
\Big\| \big\| \jb{\tau}^{-b'} \jb{\tau_2}^{-b} \| \Hb_1^{N_1, N_3}  (n,n_2,\tau, \tau_2) \|_{\l^2_{n_2} \to \l^2_n} \big\|_{L^2_{\tau, \tau_2}} \Big\|_{L^p (\O)} 
 \les p  T^\ta (N_1N_3)^{-\dl} .
\label{RT7}
\end{align}

Given dyadic $N, N_1, N_2, N_3 \ge 1$, we set 
\begin{align}
\Hb_1^{N, N_1, N_2,  N_3} (n,n_2,\tau, \tau_2)  
= \ind_{|n|\sim N} \ind_{|n_2|\sim N_2}
\Hb_1^{N_1, N_3} (n,n_2,\tau, \tau_2).
\label{RT8}
\end{align}

\noi
%where $\Hb_1^{N_1, N_3} (n,n_2,\tau, \tau_2)  $ is as in \eqref{RT4}.
By introducing  the following short-hand notation:
\begin{align}
 \Hb_1^{{\bf N}} (n,n_2,\tau, \tau_2)  
= 
\Hb_1^{N, N_1, N_2,  N_3} (n,n_2,\tau, \tau_2) , 
\label{RT8a}
\end{align}

\noi
we write 
\begin{align}
\Hb_1^{N_1, N_3}
= 
\sum_{\substack{N, N_2 \ge 1, \text{ dyadic}\\ \max (N_1, N_3) \ges N_2}} 
\Hb_1^{\bf N}
+ 
\sum_{\substack{N, N_2 \ge 1, \text{ dyadic}\\ N_2 \gg \max (N_1, N_3)}} 
\Hb_1^{\bf N}.
\label{RT8b}
\end{align}

\noi
In the following, we separately estimate
the first and second terms on the right-hand side of~\eqref{RT8b}.
Before proceeding further, 
we state a lemma
on the operator norm
of 
$\Hb_1^{\bf N}$
(possibly with further frequency localizations).

\begin{lemma}
\label{LEM:RT1a} 
Let $Q_0$ and $Q_2$ be cubes of side length $\les \max(N_1,N_3)$.
Fix small  $s> 0$.
Then, we have 
\begin{align}
\sup_{\tau,\tau_2\in \R} \Big\| \| 
 \ind_{Q_0} (n) \ind_{Q_2} (n_2)
\Hb_1^{\bf N}
(n,n_2,\tau, \tau_2)
 \|_{\l^2_{n_2} \to \l^2_n} \Big\|_{L^p(\O)} 
\les p  (N_1N_3)^{- \frac 12 + s + \eps}
\label{RT9}
\end{align}

\noi 
for any finite $p \ge 1$
and  dyadic $N, N_1,  N_2, N_3  \ge 1$, 
uniformly in $0 < T \ll 1$
and the cubes $Q_0$ and $Q_2$.

\end{lemma}

We note that the supremum over $\tau$ and $\tau_2$ 
is outside the $L^p(\O)$-norm.
In the following, we first prove \eqref{RT7}
by assuming  Lemma \ref{LEM:RT1a}.
We present a proof of  Lemma~\ref{LEM:RT1a} 
at the end of this section.

\medskip

\noi 
$\bullet$ {\bf Case 1:}
 $\max (N_1, N_3) \ges N_2$. 
\\
\indent
In this case, we have
\begin{align}
\max (N, N_1, N_2, N_3)\sim \max (N_1, N_3),
\label{RT9a}
\end{align}

\noi
since $n = n_1 - n_2 + n_3$.
Then, 
 \eqref{RT7}
 follows once we prove 
\begin{align}
\label{RT10}
\Big\| \big\| \jb{\tau}^{-b'} \jb{\tau_2}^{-b} 
\| \Hb_1^{\bf N}
 (n,n_2,\tau, \tau_2) \|_{\l^2_{n_2} \to \l^2_n} \big\|_{L^2_{\tau, \tau_2}} \Big\|_{L^p (\O)} \les p T^\ta (N_1N_3)^{-2 \dl}, 
\end{align}

\noi 
since the extra factor $(N_1N_3)^{- \dl}$
allows us to sum over dyadic $N,  N_2 \ge 1$.
Furthermore, we  claim that 
 \eqref{RT10} follows once we prove 
\begin{align}
\label{RT11}
\Big\| \big\| \jb{\tau}^{-b} \jb{\tau_2}^{-b} 
\| \Hb_1^{\bf N} (n,n_2,\tau, \tau_2) \|_{\l^2_{n_2} \to \l^2_n} \big\|_{L^2_{\tau, \tau_2}} \Big\|_{L^p (\O)} \les p (N_1N_3)^{-3\dl}.
\end{align}

\noi
(Note from \eqref{b1} that $- b = -\frac 12 - \eps < -\frac 12 < -b' = - \frac 12 + 2\eps$.)
Indeed, 
by applying  \eqref{Z0}
and
crudely estimating
the $\l^2_{n_2, n}$-norm 
by the $\l^1_{n_2, n}$-norm
with \eqref{RT8} and \eqref{RT4}, 
 we have
\begin{align}
\begin{split}
\Big\| & \big\| \jb{\tau_2}^{-b} 
\| \Hb_1^{\bf N}(n,n_2,\tau, \tau_2) \|_{\l^2_{n_2} \to \l^2_n} \big\|_{L^2_{\tau , \tau_2}} \Big\|_{L^p (\O)} \\
& \le \Big\| \big\| \jb{\tau_2}^{-b} 
\| \Hb_1^{\bf N} (n,n_2,\tau, \tau_2) \|_{\l^2_{n_2, n} } \big\|_{L^2_{\tau, \tau_2}} \Big\|_{L^p (\O)} \\
& \les N^s \bigg\| \Big\|  \jb{\tau_2}^{-b} 
 \sum_{\substack{n, n_1,n_2 \in \Z^2 \\ |n| \sim N}}
\ind_{A(n)}
\frac{|g_{n_1}  {g_{n_3} }| }{\jb{n_1}\jb{n_3}}
|\ft \eta_{_T} (\tau + \tau_2 - \phi (\bar n)) |
 \Big\|_{ L^2_{\tau, \tau_2}} \bigg\|_{L^p (\O)} \\
& \les N^s \bigg\|  \sum_{\substack{n, n_1,n_2 \in \Z^2 \\ |n| \sim N}}
\ind_{A(n)}
\frac{|g_{n_1}g_{n_3} |  }{\jb{n_1}\jb{n_3}}
\|  \jb{\tau_2}^{-b}
\ft\eta_{_T} (\tau + \tau_2 - \phi (\bar n)) 
\|_{ L^2_{\tau, \tau_2}} \bigg\|_{L^p (\O)} \\
& \les T^\frac 12 N^s \sum_{\substack{n, n_1,n_2 \in \Z^2 \\ |n| \sim N}}
\frac{\ind_{A(n)}}{\jb{n_1}\jb{n_3}}
\|g_{n_1}  g_{n_3} \|_{L^p (\O)} \les p T^\frac 12  \max(N_1,N_3)^{6}
\end{split}
\label{RT12}
\end{align}

\noi 
for small $s >  0$, 
where the set 
$A(n)$ is as in \eqref{RT12a}.
Here, 
we used \eqref{decay} and 
the fact that  $b > \frac 12$ in the penultimate step
and~\eqref{RT9a}
in the last step.
Then, \eqref{RT10} follows from interpolating \eqref{RT11} and \eqref{RT12}.

We now prove \eqref{RT11}.
Without loss of generality, assume $p \ge 2 $.
Then, 
from Minkowski's integral inequality, 
the fact that $b > \frac 12$, 
and 
 Lemma \ref{LEM:RT1a} with \eqref{RT9a}, 
 we obtain
\begin{align}
\begin{split}
\text{LHS of \eqref{RT11}}
& \le 
\Big\|  \jb{\tau}^{-b} \jb{\tau_2}^{-b} 
\big\|
\| \Hb_1^{\bf N} (n,n_2,\tau, \tau_2) \|_{\l^2_{n_2} \to \l^2_n} 
\big\|_{L^p (\O)} 
\Big\|_{L^2_{\tau, \tau_2}} \\
& \les 
\sup_{\tau, \tau_2 \in \R}
\Big\|
\| \Hb_1^{\bf N} (n,n_2,\tau, \tau_2) \|_{\l^2_{n_2} \to \l^2_n} 
\Big\|_{L^p (\O)} \\
& \les p (N_1N_3)^{-\frac 14}
\end{split}
\label{RT13}
\end{align}

\noi
\noi 
for small $s >  0$.
This proves \eqref{RT11}.

\medskip

\noi 
$\bullet$ {\bf Case 2:}
 $N_2 \gg \max (N_1, N_3)$. 
\\
\indent
Let $\TT_{12}^{N_1,N_3}$ denote the contribution
from this case 
to $\TT_{1}^{N_1,N_3}$ defined in \eqref{RT1a}.
Let $\phi(\bar n)$ be as in \eqref{phi1}.
Then, we have 
$\phi (\bar n) \sim N_2^2$
in this case.
%
%
%\vspace{1cm}
%
%Let $\TT_{2a}^{N_2,N_3}$
% denote the contribution
%from this subcase 
%to $\TT_{2}^{N_2,N_3}$ defined in \eqref{RA1a}, 
%and we directly prove
%\eqref{RA2} for $\TT_{2a}^{N_2,N_3}$.
%In this subcase, we have 
%\[
%|\phi(\bar n)| = \big||n|^2 - |n_1|^2 + |n_2|^2 - |n_3|^2\big| \gg \max (N_2, N_3)^{a_1}. 
%\]
%
%\noi
Then, by the triangle inequality, 
we have 
\begin{align}
\max (\jb{\s}, \jb{\s_1},\jb{\s_2}, \jb{\s_3}) \ges N_2^{2}\gg 
\max (N_1, N_3)^2
\label{RT13x1}
\end{align}

\noi
under $\tau = \tau_1 - \tau_2 + \tau_3$, 
where 
\begin{align}
\s = \tau + |n|^2
\qquad \text{and} \qquad \s_j = \tau_j + |n_j|^2\quad  \text{for}\ \  j = 1,2,3.
\label{mod1}
\end{align}

%Note that, by a slight modification of the proof of 
%the trilinear estimate \eqref{tri1}, we have 

\noi
From \eqref{tri3}, we have 
\begin{align}
\begin{split}
& \big\|
\NN(\Q_{N_1} u_1,
  \Sb^\perp_{10 \max (N_1, N_3)}u_2,   \Q_{N_3}u_3)\big\|_{X^{s, - b' + 2\eps_1 }_T} \\
& \quad %\hphantom{XXXXXXX}
\les 
\|u_2\|_{X^{s, b-\eps_1}_T}
 \prod_{j \in \{1, 3\} } \|\Q_{N_j} u_j\|_{X^{\eps_2, b - \eps_1}_T}
\end{split}
\label{RT13x2}
\end{align}

\noi
for any  $0 < T \le 1$, $s > 0$,  and $\eps_2 > 0$, 
provided that $\eps_1 > 0$ is sufficiently small.
Here, $\Sb^\perp_N$ is as in~\eqref{LP2} and $b$ and $b'$ are as in \eqref{b1}.
Then, from Lemma \ref{LEM:decay1}, \eqref{RT13x1},  and \eqref{RT13x2}, 
we obtain 
\begin{align*}
\|\TT_{12}^{N_1,N_3}(w)\|_{X^{s, - b'}_T}
\les T^{\eps_1} \max (N_1, N_3)^{-2\eps_1}
\|w\|_{X^{s, b}_T}
 \prod_{j \in \{1, 3\} } \|\Q_{N_j} z\|_{X^{\eps_2, b }_T}.
\end{align*}

\noi
Hence, by 
taking a supermum over $\|w\|_{X^{s, b}_T}\le 1$
and 
applying  the homogeneous linear estimate~\eqref{lin2}
(with \eqref{lin1})
and the Wiener chaos estimate (Lemma~\ref{LEM:hyp}), 
 we have
\begin{align*}
\Big\| \| \TT_{12}^{N_1,N_3} \|_{X^{s,b}_T \to 
X^{s,-b'}_T} \Big\|_{L^p (\O)}  
& \les 
T^{\eps_1} \max (N_1, N_3)^{-2\eps_1+ 4\eps_2}
\big\|\| u_0^\o\|_{H^{-\eps_2}}\big\|_{L^{2p} (\O)}^2  \\
& \les p
T^{\eps_1} \max (N_1, N_3)^{-2\eps_1+ 4\eps_2}.
\end{align*}

\noi
Therefore, 
by choosing  $0 < \eps_2 \ll \eps_1 < \eps \ll1$, 
we obtain \eqref{RT2} in this subcase.

\medskip

This concludes  the proof of Proposition \ref{PROP:RT1}.
\end{proof}

We conclude this section by presenting a proof of 
 Lemma \ref{LEM:RT1a}.

\begin{proof}[Proof of Lemma \ref{LEM:RT1a}]

Given dyadic $N, N_1, N_2, N_3 \ge 1$, write $\Hb_1^{\bf N}$
in \eqref{RT8} and \eqref{RT8a} 
as 
\begin{align}
\Hb_1^{\bf N}
 (n,n_2,\tau, \tau_2)  
= \sum_{m \in \Z} \Hb^{{\bf N}, (m)}_1 (n, n_2 )\ft \eta_{_T} (\tau + \tau_2 - m),   
\label{RTT1}
\end{align}

\noi
where $\Hb^{{\bf N}, (m)}_1$ is defined by 
\begin{align}
\begin{split}
 \Hb^{{\bf N}, (m)}_1 (n, n_2)
& =  \frac{\jb{n}^{s}}{ \jb{n_2}^{s}} 
\ind_{\substack{|n| \sim N\\|n_2|\sim N_2}}
\sum_{\substack{n + n_2 = n_1 + n_3\\ n \neq n_1,n_3}} 
\bigg(\prod_{j \in \{ 1, 3\}}
\frac{g_{n_j} }{\jb{n_j}}
\ind_{|n_j|\sim N_j}\bigg)
\ind_{ \phi (\bar n) = m}\\
& =  \frac{\jb{n}^{s}}{ \jb{n_2}^{s}} 
\sum_{n_1, n_3 \in \Z^2} 
%\prod_{j \in \{ 1, 3\}}
%\frac{1 }{\jb{n_j}}\cdot
h^{{\bf N}, (m)}(n, n_1, n_2, n_3)
%\prod_{j \in \{ 1, 3\}}
\frac{g_{n_1} g_{n_3}}{\jb{n_1}\jb{n_3}}.
\end{split}
\label{RTT2}
\end{align}

\noi
Here,  $h^{{\bf N}, (m)}$ is the base tensor defined in \eqref{baseT}. 
Hence, from \eqref{RTT1}, 
$\ft \eta_{_T}(\tau) = T \ft \eta (T\tau)$, 
Lemma~\ref{LEM:RT}\,(i) with \eqref{RTT2}, and Lemma \ref{LEM:basetensor}, we have
\[
\begin{split} 
& \sup_{\tau,\tau_2\in \R} \Big\|  \| 
 \ind_{Q_0} (n) \ind_{Q_2} (n_2)
\Hb_1^{\bf N} (n,n_2,\tau,\tau_2) \|_{\l^2_{n_2} \to \l^2_n} \Big\|_{L^p(\O)} \\
& \quad \le \sup_{\tau,\tau_2\in \R} \sum_{m \in \Z}
|\ft \eta_{_T} (\tau + \tau_2 - m)| 
\Big\| \|   \ind_{Q_0} (n) \ind_{Q_2} (n_2) \Hb_1^{{\bf N}, (m)} (n,n_2) 
\|_{\l^2_{n_2} \to \l^2_n} \Big\|_{L^p(\O)}  \\
&\quad \les p  \max(N_1, N_3)^\eps N^s N_1^{-1}N_2^{-s}  N_3^{-1} 
\sup_{m \in \Z}
\max_{(B, C)}  \|  h^{{\bf N}, (m)}  \|_{n_2 n_B  \to n n_C},  \\
& \quad \les p 
 (N_1N_3)^{-\frac12 + s+  2\eps}, 
\end{split}
\]

\noi
where 
%$N_{\max}$ is as in \eqref{ord1} and 
$(B, C)$ is a partition of $\{1, 3\}$.
Namely, the maximum is taken over 
\[n_1 n_2 n_3  \to n , \qquad  
n_2 n_3 \to n n_1, 
\qquad n_1 n_2 \to nn_3, 
\qquad \text{and}\qquad 
n_2 \to nn_1 n_3. \]

\noi
This proves  \eqref{RT9}.
\end{proof}

\section{Random tensor term II}
\label{SEC:RT2}

In this section, we estimate
the  random tensor term:
% $\NN(w, z, z)$. 
%Define the random operator $\TT_2$
%by setting
\begin{align}\label{RA0}
\TT_2(w) =  \NN (w,z, z), 
\end{align}

\noi
where $z$ denotes the random linear solution in \eqref{lin1}.
Then, our main goal is to prove the following proposition.

\begin{proposition}
\label{PROP:RT2}
Given small   $s > 0$ and $\eps > 0$, 
 there exists $\ta > 0$
such that 
\begin{align*}
\Big\| \| \TT_2 \|_{X^{s,b}_T \to 
X^{s,-b'}_T} \Big\|_{L^p (\O)} 
\les p  T^\ta
%\label{RA0a}
\end{align*}

\noi
for any finite $p \ge 1$ and $0 < T\ll 1$, where 
$b = \frac 12 + \eps$ and $b' = \frac 12 - 2\eps$
are as in \eqref{b1}.
In particular, 
 the conclusion of Proposition \ref{PROP:1}
holds
for the random tensor term $\NN(w, z, z)$.

\end{proposition}

\begin{proof}

Given dyadic $N_2, N_3 \ge 1$, 
set 
\begin{align}
\TT_2^{N_2,N_3} (w) =  \NN(w, \Q_{N_2} z, \Q_{N_3} z), 
\label{RA1a}
\end{align}

\noi
where $\Q_{N_j}$ denotes the Littlewood-Paley projector.
Then, as in the proof of Proposition \ref{PROP:RT1}, 
it suffices to prove that 
there exist small $\ta, \dl > 0$ such that 
\begin{align}
\Big\| \| \TT_2^{N_2,N_3} \|_{X^{s,b}_T \to 
X^{s,-b'}_T} \Big\|_{L^p (\O)} 
\les p  T^\ta (N_2N_3)^{-\dl} 
\label{RA2}
\end{align}

\noi
for any finite $p \ge 1$ and $0 < T\ll 1$.

Proceeding as in \eqref{RT3} with 
\eqref{RA1a}, \eqref{RA0},  and \eqref{non1}, we have 
\begin{align*}
%\begin{split}
 \F_{x, t} &
\big(\eta_{_T} \TT_2^{N_2,N_3} (w)\big)
 (n,\tau -|n|^2) \\
&= \jb{n}^{-s} \sum_{n_1\in \Z^2 } \jb{n_1}^s 
\int_\R 
\Hb_2^{N_2, N_3} (n,n_1,\tau, \tau_1) 
 \ft w(n_1, \tau_1 - |n_1|^2) 
d\tau_1,
%\end{split}
%\label{RA3}
\end{align*} 

\noi
where 
\begin{align}
\begin{split}
& \Hb_2^{N_2, N_3} (n,n_1,\tau, \tau_1)  \\
& \quad =  \frac{\jb{n}^{s}}{ \jb{n_1}^{s}} 
\sum_{\substack{n - n_1 = - n_2 + n_3\\ n \neq n_1,n_3}} 
\bigg(\prod_{j = 2}^3 
\frac{g^*_{n_j} }{\jb{n_j}}
\ind_{|n_j|\sim N_j}\bigg)
\ft \eta_{_T} (\tau - \tau_1 - \phi (\bar n))  
\end{split}
\label{RA4}
\end{align}

\noi
with  $\phi (\bar n)$ as in \eqref{phi1}
and $g^*_{n_j}$ as in \eqref{g1}.
Then, 
by repeating the computations in 
\eqref{RT5} and~\eqref{RT6}, 
we see that~\eqref{RA2} follows once we prove
\begin{align}
\Big\| \big\| \jb{\tau}^{-b'} \jb{\tau_1}^{-b} \| \Hb_2^{N_2, N_3}  (n,n_1,\tau, \tau_1) \|_{\l^2_{n_1} \to \l^2_n} \big\|_{L^2_{\tau, \tau_1}} \Big\|_{L^p (\O)} 
 \les p  T^\ta (N_2N_3)^{-\dl} .
\label{RA7}
\end{align}

Given dyadic $N, N_1, N_2, N_3 \ge 1$, we set 
\begin{align*}
%\begin{split}
 \Hb_2^{{\bf N}} (n,n_1,\tau, \tau_1)  
& = \Hb_2^{N, N_1, N_2,  N_3} (n,n_1,\tau, \tau_1)  \\
& = \ind_{|n|\sim N} \ind_{|n_1|\sim N_1}
\Hb_2^{N_2, N_3} (n,n_1,\tau, \tau_1).
%\end{split}
%\label{RA8}
\end{align*}

\noi
Then, by writing 
\begin{align}
\Hb_2^{N_2, N_3}
= 
\sum_{\substack{N, N_1 \ge 1, \text{ dyadic}\\ \max (N_2, N_3) \ges N_1}} 
\Hb_2^{\bf N}
+ 
\sum_{\substack{N, N_1 \ge 1, \text{ dyadic}\\ N_1 \gg \max (N_2, N_3)}} 
\Hb_2^{\bf N}, 
\label{RA8b}
\end{align}

\noi
 we separately estimate
the first and second terms on the right-hand side of~\eqref{RA8b}.

\medskip

\noi 
$\bullet$ {\bf Case 1:}
 $\max (N_2, N_3) \ges N_1$. 
\\
\indent
In this case, a straightforward modification of the argument in 
Case 1 of  the proof of Proposition \ref{PROP:RT1}
yields \eqref{RA7}, and thus we omit details.

\medskip

\noi 
$\bullet$ {\bf Case 2:}
 $N_1 \gg \max (N_2, N_3)$. 
\\
\indent
In this case, we have $N \sim N_1
\gg \max (N_2, N_3)$.
Unlike Case 2 in the proof of Proposition~\ref{PROP:RT1}, 
there is no effective lower bound on $|\phi(\bar n)|$.
We separately consider the following subcases:

\smallskip

\begin{itemize}
\item[(2.a)]
$\big| |n|^2 - |n_1|^2 \big| \gg \max (N_2, N_3)^{a_1}$,

\smallskip

\item[(2.b)]
$\big| |n|^2 - |n_1|^2 \big| \les \max (N_2, N_3)^{a_1}$

\end{itemize}

\smallskip

\noi
for some fixed  $a_1 > 2$.
%\ge 1$
%as in Lemma \ref{LEM:RT}\,(iii) (and Lemma \ref{LEM:DH}).

%\gg 1$ is 
%
%$a_1 > 2$.

\medskip

\noi
$\pmb\circ$ {\bf Subcase 2.a:}
$\big| |n|^2 - |n_1|^2 \big| \gg \max (N_2, N_3)^{a_1}$.
\\
\indent
Let $\TT_{2a}^{N_2,N_3}$
 denote the contribution
from this subcase 
to $\TT_{2}^{N_2,N_3}$ defined in \eqref{RA1a}.
We will directly prove
\eqref{RA2} for $\TT_{2a}^{N_2,N_3}$.
In this subcase, we have 
\[
|\phi(\bar n)| = \big||n|^2 - |n_1|^2 + |n_2|^2 - |n_3|^2\big| \gg \max (N_2, N_3)^{a_1}. 
\]

\noi
Then, by the triangle inequality, 
we have 
\begin{align}
\max (\jb{\s}, \jb{\s_1},\jb{\s_2}, \jb{\s_3}) \gg \max (N_2, N_3)^{a_1}
\label{RA9}
\end{align}

\noi
under $\tau = \tau_1 - \tau_2 + \tau_3$, 
where $\s$ and $\s_j$ are as in \eqref{mod1}.
Then, proceeding as in Case 2 of the proof of 
 Proposition \ref{PROP:RT1}
 with \eqref{RA9}, we have
\begin{align*}
\Big\| \| \TT_{2b}^{N_2,N_3} \|_{X^{s,b}_T \to 
X^{s,-b'}_T} \Big\|_{L^p (\O)}  
& \les 
T^{\eps_1} \max (N_2, N_3)^{-a_1\eps_1+ 4\eps_2}
\big\|\| u_0^\o\|_{H^{-\eps_2}}\big\|_{L^{2p} (\O)}^2  \\
& \les p
T^{\eps_1} \max (N_2, N_3)^{-a_1\eps_1+ 4\eps_2}.
\end{align*}

\noi
Therefore, 
by choosing  $0 < \eps_2 \ll \eps_1 < \eps \ll1$, 
we obtain \eqref{RA2} in this subcase.

\medskip

\noi
$\pmb\circ$ {\bf Subcase 2.b:}
$\big| |n|^2 - |n_1|^2 \big| \les \max (N_2, N_3)^{a_1}$.
\\
\indent
Let  $ \Hb_{2b}^{N_2, N_3}$ denote the contribution
from  this subcase
to $\Hb_2^{N_2, N_3}$ in \eqref{RA4}, namely,  
\begin{align}
\begin{split}
 \Hb_{2b}^{N_2, N_3}(n,n_1,\tau, \tau_1)
&   = 
 \frac{\jb{n}^{s}}{ \jb{n_1}^{s}} \sum_{\substack{N, N_1 \ge 1, \text{ dyadic}\\ N_1 \gg \max (N_2, N_3)\\
| |n|^2 - |n_1|^2 | \les \max (N_2, N_3)^{a_1}}} 
\ind_{|n|\sim N} \ind_{|n_1|\sim N_1}
\\
%\Hb_2^{\bf N}(n,n_1,\tau, \tau_1)\\
& \hphantom{XX}
\sum_{\substack{n - n_1 = - n_2 + n_3\\ n \neq n_1,n_3}} 
\bigg(\prod_{j = 2}^3 
\frac{g^*_{n_j} }{\jb{n_j}}
\ind_{|n_j|\sim N_j}\bigg)
\ft \eta_{_T} (\tau - \tau_1 - \phi (\bar n))  .
\end{split}
\label{RA12}
\end{align}

\noi
Note that, under  the condition $ N_1\gg \max(N_2, N_3)$,
 the condition $n\ne n_3$ (i.e.~$n_1 \ne n_2$)
in~\eqref{RA12} is vacuous.
Since we have $|n| \sim |n_1|$ on the support of 
 $\Hb_{2b}^{N_2, N_3}$, we have 
\begin{align}
 \|  \Hb_{2b}^{N_2, N_3}  (n,n_1,\tau, \tau_1) \|_{\l^2_{n_1} \to \l^2_n} 
 \les
  \|  \wt \Hb_{2b}^{N_2, N_3}  (n,n_1,\tau, \tau_1) \|_{\l^2_{n_1} \to \l^2_n}
\label{RA13a}
\end{align}

\noi
for each $\tau, \tau_2 \in \R$, 
where 
$\wt  \Hb_{2b}^{N_2, N_3}$ is defined by 
\begin{align*}
\wt  \Hb_{2b}^{N_2, N_3}
(n,n_1,\tau, \tau_1)
= \frac{\jb{n_1}^s}{\jb{n}^s} \Hb_{2b}^{N_2, N_3}
(n,n_1,\tau, \tau_1).
\end{align*}

\noi
Then, 
by writing 
\begin{align*}
\wt  \Hb_{2b}^{N_2, N_3}
(n,n_1,\tau, \tau_1)
 =  \sum_{\substack{n_2, n_3 \in \Z^2\\ |n_2|\sim N_2\\|n_3|\sim N_3}}
\wt h^{\tau, \tau_1}_{n, n_1, n_2, n_3}
\cj{g_{n_2}}g_{n_3}, 
\end{align*}

\noi
we see that, 
for fixed $\tau$ and $\tau_1$, 
the tensor 
$\wt h^{\tau, \tau_1}_{n, n_1, n_2, n_3}$
depends only on $n - n_1$, $\big||n|^2 - |n_1|^2 \big|$, 
and $(n_2, n_3)$
with the restriction $|n|, |n_1| \gg \max(N_2, N_3)$
on the input and output frequencies, 
to which Lemma \ref{LEM:RT}\,(iii) can be applied.

By an interpolation argument
 as in 
Case 1 of the proof of Proposition \ref{PROP:RT1}
and \eqref{RA13a}, 
we claim that~\eqref{RA7} in this subcase follows
once we prove
\begin{align}
\Big\| \big\| \jb{\tau}^{-b} \jb{\tau_1}^{-b}
 \| \wt  \Hb_{2b}^{N_2, N_3}  (n,n_1,\tau, \tau_1) \|_{\l^2_{n_1} \to \l^2_n} \big\|_{L^2_{\tau, \tau_1}} \Big\|_{L^p (\O)} 
 \les p   (N_2N_3)^{-\dl} .
\label{RA14}
\end{align}

\noi
where the power on $\jb{\tau}$ is replaced by $-b$.

As in \eqref{RTX1}, we define $  \wt \Hb_2^{N_2, N_3, m}$ by setting
\begin{align*}
\wt \Hb_2^{N_2, N_3, m}(n' ,n_1 ',\tau, \tau_1) 
=   \wt \Hb_2^{N_2, N_3}(n'+m ,n_1'+m ,\tau, \tau_1) 
\cdot   \ind_{\substack{|n'|\les  \max (N_2, N_3)\\|n_1'|\les  \max (N_2, N_3)}}.
\end{align*}

\noi
Then, 
arguing 
as in the proof of Lemma \ref{LEM:RT}\,(iii)
with Lemma \ref{LEM:DH}
(see \eqref{XY6}), 
we see that there exists $J =O(\max (N_2, N_3)^{a_2})$
and $\{m_j \}_{j = 1}^J\subset \Z^2$
with $|m_j| \gg \max (N_2, N_3)$ 
such that 
\begin{align}
\begin{split}
\Big\| & \big\| \jb{\tau_1}^{-b} 
\|  \wt \Hb_2^{N_2, N_3}(n,n_1,\tau, \tau_1) \|_{\l^2_{n_1} \to \l^2_{n}} \big\|_{L^2_{\tau , \tau_1}} \Big\|_{L^p (\O)} \\
& \les \Big\|  \big\| \jb{\tau_1}^{-b} 
\sup_{m \in \Z^2} \|\wt   \Hb_2^{N_2, N_3, m}(n',n_1' ,\tau, \tau_1) 
\|_{\l^2_{n_1'} \to \l^2_{n'}} \big\|_{L^2_{\tau , \tau_1}} \Big\|_{L^p (\O)} \\
& \les \Big\|  \big\| \jb{\tau_1}^{-b} 
 \sup_{|m| \les \max (N_2, N_3)} \| \wt  \Hb_2^{N_2, N_3, m}(n ',n_1' ,\tau, \tau_1) 
\|_{\l^2_{n_1'} \to \l^2_{n'}} \big\|_{L^2_{\tau , \tau_1}} \Big\|_{L^p (\O)} \\
& \quad 
+ \Big\|  \big\| \jb{\tau_1}^{-b} 
\sup_{j  = 1, \dots, J} \| \wt  \Hb_2^{N_2, N_3, m_j}(n ',n_1' ,\tau, \tau_1) 
\|_{\l^2_{n_1'} \to \l^2_{n'}} \big\|_{L^2_{\tau , \tau_1}} \Big\|_{L^p (\O)} \\
& \les
\max (N_2, N_3)\\
& \quad \quad 
\times  \sup_{|m| \les  \max (N_2, N_3)}
 \Big\| \big\| \jb{\tau_1}^{-b} 
\| \wt  \Hb_2^{N_2, N_3, m} (n',n_1',\tau, \tau_1) \|_{\l^2_{n_1', n'}} \big\|_{L^2_{\tau, \tau_1}} \Big\|_{L^p (\O)}\\
& \quad 
+ 
\max (N_2, N_3)^\frac{a_2}{2}
\sup_{j  = 1, \dots, J}
 \Big\| \big\| \jb{\tau_1}^{-b} 
\| \wt  \Hb_2^{N_2, N_3, m_j} (n',n_1',\tau, \tau_1) \|_{\l^2_{n_1', n'}} \big\|_{L^2_{\tau, \tau_1}} \Big\|_{L^p (\O)}
\end{split}
\label{RA14a}
\end{align}

\noi
for any $p \ge 2$.
Then, proceeding as in \eqref{RT12}, 
we obtain
\begin{align}
\Big\|  \big\| \jb{\tau_1}^{-b} 
\| \wt  \Hb_2^{N_2, N_3}(n,n_1,\tau, \tau_1) \|_{\l^2_{n_1} \to \l^2_n} \big\|_{L^2_{\tau , \tau_1}} \Big\|_{L^p (\O)} 
 \les p T^\frac 12  \max(N_2,N_3)^{a_3}
\label{RA15}
\end{align}

\noi 
for some $a_3 \gg 1$.
Then, \eqref{RA7} in this subcase follows from interpolating \eqref{RA14} and \eqref{RA15}. 

It remains to prove \eqref{RA14}.
Since $b > \frac 12$, we can proceed as in 
\eqref{RT13}
by applying 
Lemma~\ref{LEM:RT}\,(iii)
(instead of 
Lemma~\ref{LEM:RT}\,(i)
used in 
the proof of Lemma \ref{LEM:RT1a} presented at the end of Section \ref{SEC:RT1}).
This proves \eqref{RA14}
(with $\dl = \frac 12 - \eps$).

\medskip

This concludes  the proof of Proposition \ref{PROP:RT2}.
\end{proof}

\section{Random tensor term III}
\label{SEC:RT3}

In this section, we estimate
the  random tensor term:
\begin{align}\label{RB0}
\TT_3(w_2, w_3) =  \NN (z, w_2,w_3), 
\end{align}

\noi
where $z$ denotes the random linear solution in \eqref{lin1}.
Then, our main goal is to prove the following proposition.

\begin{proposition}
\label{PROP:RT3}
Given small   $s > 0$ and $\eps > 0$, 
 there exists $\ta > 0$
such that 
\begin{align*}
\Big\| \| \TT_3 \|_{X^{s,b}_T\times X^{s,b}_T  \to X^{s,-b'}_T} 
\Big\|_{L^p (\O)} 
\les p^\frac 12   T^\ta
%\label{RB0a}
\end{align*}

\noi
for any finite $p \ge 1$ and $0 < T\ll 1$, where 
$b = \frac 12 + \eps$ and $b' = \frac 12 - 2\eps$
are as in \eqref{b1}.
In particular, 
 the conclusion of Proposition \ref{PROP:1}
holds
for the random tensor term $\NN(z, w, w)$.

\end{proposition}

\begin{proof}
Given dyadic $N_1 \ge 1$, define
\begin{align}
\TT_3^{N_1} (w_2, w_3) =  \NN(\Q_{N_1} z, w_2, w_3).
\label{RB1}
\end{align}

\noi
Then, it suffices to prove that 
there exist 
 small $\ta, \dl > 0$ such that 
\begin{align}
\Big\| \| \TT_3^{N_1} 
\|_{X^{s,b}_T\times X^{s,b}_T  \to X^{s,-b'}_T} 
\Big\|_{L^p (\O)} 
\les p^\frac 12   T^\ta N_1^{-\dl} 
\label{RB2}
\end{align}

\noi
for any finite $p \ge 1$ and $0 < T\ll 1$.
In the following, we use $n_1$, $n_2$, and $n_3$
to denote
the spatial frequencies of $\Q_{N_1} z$, $w_2$, and $w_3$, 
respectively.

\medskip

\noi 
$\bullet$ {\bf Case 1:}
$N_1   \les \min(|n_2|, |n_3|)$.
\\
\indent
Let $\TT_{31}^{N_1}$ denote the contribution
from this case 
to $\TT_{3}^{N_1}$ defined in \eqref{RB1}.
Without loss of generality, assume that 
$N_1\les |n_2| \le |n_3|$.
Then, 
from Lemma \ref{LEM:decay1} and \eqref{tri3} 
with small $\eps_1, \eps_2 > 0$, we have 
\begin{align*}
\big\|
\NN( \Q_{N_1}  u_1, u_2,   u_3)\big\|_{X^{s, - b' }_T} 
& \les 
T^{\eps_1}
\sum_{\substack{N_2 \ge 1\\\text{dyadic}}}
\ind_{N_2 \ges N_1}
\big\|
\NN( \Q_{N_1} u_1, \Q_{N_2}u_2,   u_3)\big\|_{X^{s, - b'+\eps_1 }_T} \\
& 
\les 
T^{\eps_1}
\sum_{\substack{N_2 \ge 1\\\text{dyadic}}}
\ind_{N_2 \ges N_1}
 \|\Q_{N_1}u_1\|_{X^{\eps_2, b }_T}
  \|\Q_{N_2}u_2\|_{X^{\eps_2, b }_T}
\|u_3\|_{X^{s, b}_T}\\
& 
\les T^{\eps_1} 
N_1^{-\eps_2}
\|u_1\|_{X^{- \eps_2, b }_T}
  \|u_2\|_{X^{s, b }_T}
\|u_3\|_{X^{s, b}_T}
\end{align*}

\noi
for any  $0 < T \le 1$, 
provided that $s \ge  4\eps_2$.
In particular, we have 
\begin{align*}
\|\TT_{31}^{N_1}(w_2, w_3)\|_{X^{s, - b'}_T}
\les T^{\eps_1} 
N_1^{-\eps_2}
 \| z\|_{X^{-\eps_2, b }_T}
\|w_2\|_{X^{s, b}_T}
\|w_3\|_{X^{s, b}_T}
\end{align*}

\noi
Hence, by 
taking a supermum over $\|w_j\|_{X^{s, b}_T}\le 1$, $j = 2, 3$, 
and applying  the homogeneous linear estimate~\eqref{lin2}
(with \eqref{lin1})
and the Wiener chaos estimate (Lemma~\ref{LEM:hyp}), 
 we have
\begin{align*}
\Big\| \| \TT_{31}^{N_1}
\|_{X^{s,b}_T\times X^{s,b}_T  \to X^{s,-b'}_T} 
\Big\|_{L^p (\O)}  
& \les 
T^{\eps_1} 
N_1^{-\eps_2}
\big\|\| u_0^\o\|_{H^{-\eps_2}}\big\|_{L^{p} (\O)}\\
&  \les p^\frac 12 
T^{\eps_1}N_1^{-\eps_2}, 
\end{align*}

\noi
which yields 
\eqref{RB2} in this case.

\medskip

\noi 
$\bullet$ {\bf Case 2:}
$|\phi(\bar n)| \gg  N_1^2$.
\\
\indent
Let $\TT^{N_1}_{32}$ denote the contribution
from this case 
to $\TT_{3}^{N_1}$ defined in \eqref{RB1}.
Then, 
proceeding as in Case 2 of the proof of Proposition \ref{PROP:RT1}, 
we have
\begin{align*}
\|\TT_{32}^{N_1}(w_2, w_3)\|_{X^{s, - b'}_T}
\les T^{\eps_1} N_1^{-2\eps_1}
 \| z\|_{X^{-\eps_2, b }_T}
\|w_2\|_{X^{s, b}_T}
\|w_3\|_{X^{s, b}_T}.
\end{align*}

\noi
Then,  by 
taking a supermum over $\|w_j\|_{X^{s, b}_T}\le 1$, $j = 2, 3$, 
and applying  the homogeneous linear estimate~\eqref{lin2}
(with \eqref{lin1})
and the Wiener chaos estimate (Lemma~\ref{LEM:hyp}), 
 we obtain~\eqref{RB2} in this case.

\medskip

Before proceeding further, let us carry out 
a basic reduction.
From \eqref{RB1} with \eqref{RB0} and~\eqref{non1}, 
we have 
\begin{align*}
%\begin{split}
& \F \big(\eta_{_T}  
 \TT_{3}^{N_1}(w_2, w_3) \big)
 (n,\tau - |n|^2) \\ 
& \quad = \jb{n}^{-s}
 \sum_{n_2,n_3 \in \Z^2}
 \jb{n_2}^s  \jb{n_3}^s \int_{\R^2}
\Hb_3^{N_1} (n,n_2,n_3,\tau,\tau_2,\tau_3)  \\
& \hphantom{XXXXXXXX} \times
\cj{\ft w_2(n_2, \tau_2 - |n_2|^2)}  \ft w_3(n_3, \tau_3  - |n_3|^2) d\tau_2 d\tau_3\\
%\end{split}
%\label{RB3}
\end{align*}

\noi
where\footnote{Note that there is at most one term in the sum over $n_1$.} 
\begin{align}
\begin{split}
\Hb_3^{N_1} (n,n_2,n_3,\tau,\tau_2,\tau_3)  
& = 
\frac{\jb{n}^s}{\jb{n_2}^s \jb{n_3}^s} \\
& \quad \times  \sum_{\substack{n + n_2 - n_3 = n_1\\
n_1 \ne n,  n_2}}
\frac{g_{n_1}  }{\jb{n_1}}  \ind_{|n_1|\sim N_1}
 \ft \eta_{_T}  (\tau + \tau_2 - \tau_3 - \phi (\bar n))
\end{split}
\label{RB4}
\end{align}

\noi
with  $\phi (\bar n)$ as in \eqref{phi1}.
Arguing as in \eqref{RT5}, we see that \eqref{RB2}
follows once we prove 
\begin{align}
\begin{split}
& 
\Big\|
\big\|  \jb{\tau}^{-b'} \jb{\tau_2}^{-b} \jb{\tau_3}^{-b}
\| \Hb_3^{N_1}(n,n_2,n_3,\tau,\tau_2,\tau_3) \|_{\l^2_{n_2, n_3} \to \l^2_n}  \big\|_{L^2_{\tau,\tau_2,\tau_3}}
\Big\|_{L^p(\O)}\\
& \quad 
\les p^\frac 12   T^\ta N_1^{-\dl},  
\end{split}
\label{RB5}
\end{align}

\noi 
where $b$ and $b'$ are as in \eqref{b1}.  See also Remark \ref{REM:tensor}.

Given dyadic $N, N_1, N_2, N_3 \ge 1$, 
we set 
\begin{align}
\begin{split}
\Hb_3^{\bf N}(n,n_2,n_3,\tau,\tau_2,\tau_3)
& = 
\Hb_3^{N, N_1, N_2, N_3}(n,n_2,n_3,\tau,\tau_2,\tau_3)\\
& = \ind_{|n|\sim N}
\bigg(\prod_{j = 2}^3
\ind_{|n_j|\sim N_j}
\bigg)\Hb_3^{N_1}(n,n_2,n_3,\tau,\tau_2,\tau_3).
\end{split}
\label{RB6}
\end{align}

\noi
Then, by writing 
\begin{align}
\Hb_3^{N_1}
= 
\sum_{\substack{N, N_2, N_3 \ge 1, \text{ dyadic}\\ N_1 \ges \max (N_2, N_3)}} 
\Hb_3^{\bf N}
+ 
\sum_{\substack{N, N_2, N_3 \ge 1, \text{ dyadic}\\ N_1 \ll \max (N_2, N_3)}} 
\Hb_3^{\bf N}, 
\label{RB7}
\end{align}

\noi
 we separately estimate
the first and second terms on the right-hand side of~\eqref{RB7}
(after removing the contributions from Cases 1 and 2).

We now state a lemma whose proof is presented at the end of this section.

\begin{lemma}\label{LEM:RT3a}
Given dyadic $N, N_1, N_2, N_3 \ge 1$, 
suppose that $N_1 \ges N_{\med}$, where $N_{\med}$ is as in~\eqref{ord1}.
Let $Q_0$, $Q_2$, and $Q_3$
be cubes of side length $\les N_1$.
Fix small $s > 0$.
Then, we have 
\begin{align}
\begin{split}
&  \sup_{\tau,\tau_2,\tau_3\in \R} 
 \Big\| \| 
\ind_{Q_0} (n) \ind_{Q_2} (n_2) \ind_{Q_3} (n_3) \\
&  \hphantom{XXXXX}
\times
\Hb_3^{\bf N}(n,n_2,n_3,\tau,\tau_2,\tau_3) 
 \|_{\l^2_{n_2, n_3} \to \l^2_n} \Big\|_{L^p(\O)}  \les 
p^\frac 12  N_1^{- s + \eps}
\end{split}
\label{RB10}
\end{align}

\noi 
for any finite $p \ge 1$
and  dyadic $N, N_1,  N_2, N_3  \ge 1$, 
uniformly in $0 < T \ll 1$
and the cubes $Q_0$, $Q_2$, and $Q_3$.
\end{lemma}

\medskip

\noi 
$\bullet$ {\bf Case 3:}
$N_1\ges \max(N_2, N_3)$.
\\
\indent
We proceed by an interpolation argument as in Case 1 of the proof of Proposition \ref{PROP:RT1}.
Proceeding as in \eqref{RT12}
with 
\eqref{Z0}, \eqref{RB6}, \eqref{RB4}, and 
the Wiener chaos estimate (Lemma \ref{LEM:hyp}), we have 
\begin{align}
\begin{split}
& \Big\|  
\big\| 
\jb{\tau_2}^{-b} \jb{\tau_3}^{-b}  
\| \Hb_3^{\bf N} (n,n_2,n_3,\tau,\tau_2,\tau_3) \|_{\l^2_{n_2, n_3} \to \l^2_n}  \big\|_{L^2_{\tau,\tau_2,\tau_3}} 
\Big\|_{L^p(\O)} 
\\
& \quad \le
 \Big\|  
\big\| 
\jb{\tau_2}^{-b} \jb{\tau_3}^{-b}  
\| \Hb_3^{\bf N} (n,n_2,n_3,\tau,\tau_2,\tau_3) \|_{\l^2_{n_2, n_3, n}}  \big\|_{L^2_{\tau,\tau_2,\tau_3}} 
\Big\|_{L^p(\O)} \\
&\quad  \les 
p^\frac 12
\sum_{\substack{n,n_1,n_2,n_3 \in \Z^2\\|n|\sim N}} 
\ind_{A(n)}
\frac{N^s}{N_1 N_2^{s} N_3^{s}} 
\| \jb{\tau_2}^{-b} \jb{\tau_3}^{-b}   \ft \eta_{_T} (\tau + \tau_2 - \tau_3 - \phi (\bar n)) \|_{L^2_{\tau,\tau_2,\tau_3}} 
\\
& \quad \les  p^\frac 12 T^\frac 12  N_1^{6}
\end{split}
\label{RB8}
\end{align}

\noi
for small $s > 0$, 
where $A(n)$ is as in \eqref{RT12a}.
Hence, by interpolating with \eqref{RB8}, we see that~\eqref{RB5} follows once we prove
\begin{align}
\begin{split}
& \Big\|
\big\|  \jb{\tau}^{-b} \jb{\tau_2}^{-b} \jb{\tau_3}^{-b}
\| \Hb_3^{N_1}(n,n_2,n_3,\tau,\tau_2,\tau_3) \|_{\l^2_{n_2, n_3} \to \l^2_n}  \big\|_{L^2_{\tau,\tau_2,\tau_3}}
\Big\|_{L^p(\O)}\\
& \quad \les p^\frac 12    N_1^{-2\dl}, 
\end{split}
\label{RB9}
\end{align}

\noi
since the extra factor $N_1^{-\dl}$
allows us to sum over dyadic $N, N_2, N_3 \ge 1$.

Recalling $ b > \frac 12$, 
the desired bound \eqref{RB9} in this case 
follows from Minkowski's integral inequality and
Lemma~\ref{LEM:RT3a};
see \eqref{RT13}.

\medskip

\noi 
$\bullet$ {\bf Case 4:}
$N_1  \ll \max(N_2, N_3)$.
\\
\indent
In view of Case 1, we have $N_1 \gg \min (N_2, N_3)$.
We separately consider the following subcases:
\[ \text{(4.a)}
\  N\sim N_2 \gg N_1 \gg N_3
\qquad \text{and}
\qquad 
\text{(4.b)}
\  N\sim N_3 \gg N_1 \gg N_2.
\]

%\medskip

\noi
$\pmb\circ$ {\bf Subcase 4.a:}
 $N\sim N_2 \gg N_1 \gg N_3$.
\\
\indent
In this case, we have $|\phi(\bar n)| \sim N^2 \gg N_1^2$, 
which is already treated in Case 2.

\medskip

\noi
$\pmb\circ$ {\bf Subcase 4.b:}
 $N\sim N_3 \gg N_1 \gg N_2$.
\\
\indent
In view of Case 2, we may assume that 
$|\phi(\bar n)| \les  N_1^2$.
Since $\big||n_1|^2 - |n_2|^2\big| \les N_1^2$, 
we obtain
$\big||n|^2 - |n_3|^2\big| \les N_1^2$
in this case, 
and thus we can proceed as in Subcase 2.b of the proof 
of Proposition \ref{PROP:RT2}.

Let  $ \Hb_{3b}^{N_1}$ denote the contribution
from  this subcase
to $\Hb_{3}^{N_1}$ in \eqref{RB4}, namely,  
\begin{align}
\begin{split}
\Hb_{3b}^{N_1} (n,n_2,n_3,\tau,\tau_2,\tau_3)  
& = 
\frac{\jb{n}^s}{\jb{n_2}^s \jb{n_3}^s} 
\sum_{\substack{N, N_2, N_3 \ge 1, \text{ dyadic}\\
N\sim N_3 \gg N_1 \gg N_2\\
||n|^2 - |n_3|^2| \les N_1^2}} 
\ind_{|n|\sim N}
\bigg(\prod_{j = 2}^3\ind_{|n_j|\sim N_j}\bigg)
\\
& \quad \times  
\sum_{\substack{n + n_2 - n_3 = n_1\\n_1 \ne n, n_2}}
\frac{g_{n_1}  }{\jb{n_1}}  \ind_{|n_1|\sim N_1}
 \ft \eta_{_T}  (\tau + \tau_2 - \tau_3 - \phi (\bar n)).
\end{split}
\label{RB11}
\end{align}

\noi
Note that the condition $n_1 \ne n, n_2$ is vacuous since $N \gg N_1 \gg N_2$.
Since we have $|n| \sim |n_3|$ on the support of 
 $\Hb_{3b}^{N_1}$, we have 
\begin{align}
 \|   \Hb_{3b}^{N_1} (n,n_2,n_3,\tau,\tau_2,\tau_3)   \|_{\l^2_{n_2, n_3} \to \l^2_n} 
 \les
 \| \wt   \Hb_{3b}^{N_1} (n,n_2,n_3,\tau,\tau_2,\tau_3)   \|_{\l^2_{n_2, n_3} \to \l^2_n} 
\label{RB13a}
\end{align}

\noi
for each $\tau, \tau_2, \tau_3 \in \R$, 
where 
$ \wt  \Hb_{3b}^{N_1} (n,n_2,n_3,\tau,\tau_2,\tau_3)  $ is defined by 
\begin{align*}
\wt   \Hb_{3b}^{N_1} (n,n_2,n_3,\tau,\tau_2,\tau_3)  
= \frac{\jb{n_3}^s}{\jb{n}^s}
  \Hb_{3b}^{N_1} (n,n_2,n_3,\tau,\tau_2,\tau_3)  .
\end{align*}

\noi
Then, 
 by writing 
\begin{align*}
\wt   \Hb_{3b}^{N_1} (n,n_2,n_3,\tau,\tau_2,\tau_3)  
 =  \sum_{\substack{n_1 \in \Z^2\\ |n_1|\sim N_1}}
\wt  h^{\tau, \tau_2, \tau_3}_{n, n_1, n_2, n_3}
g_{n_1}, 
\end{align*}

\noi
we see that, 
for fixed $\tau$, $\tau_2$, and $\tau_3$, 
the tensor 
$\wt h^{\tau, \tau_2, \tau_3}_{n, n_1, n_2, n_3}$
depends only on $n - n_3$, $\big||n|^2 - |n_3|^2 \big|$, 
and $(n_1, n_2)$ 
with 
the restriction 
$|n|, |n_3| \gg N_1$
on the input and output frequencies, 
to which Lemma \ref{LEM:RT}\,(iii) can be applied.

By an interpolation argument
 as before with \eqref{RB13a},  
we claim that~\eqref{RB5} in this subcase follows
once we prove
\begin{align}
\begin{split}
& \Big\|\big\|  \jb{\tau}^{-b} \jb{\tau_2}^{-b} \jb{\tau_3}^{-b}
\| \wt  \Hb_{3b}^{N_1}(n,n_2,n_3,\tau,\tau_2,\tau_3) \|_{\l^2_{n_2, n_3} \to \l^2_n}  \big\|_{L^2_{\tau,\tau_2,\tau_3}}
\Big\|_{L^p(\O)}\\
& \quad \les p^\frac 12    N_1^{-\dl},  
\end{split}
\label{RB14}
\end{align}

\noi
where the power on $\jb{\tau}$ is replaced by $-b$.

As in \eqref{RTX1}, we define $ 
\wt  \Hb_{3b}^{N_1, m}$ by setting
\begin{align*}
\wt  \Hb_{3b}^{N_1, m}(n',n_2,n_3',\tau,\tau_2,\tau_3)
=  
\wt  \Hb_{3b}^{N_1}(n' + m ,n_2,n'_3+m,\tau,\tau_2,\tau_3)
\cdot   \ind_{\substack{|n'|\les  N_1 \\|n_3'|\les N_1}}.
\end{align*}

\noi
Then, 
proceeding as in \eqref{RA14a}
with Lemma \ref{LEM:DH}
(see also \eqref{XY6}), 
we see that there exists $J =O(N_1^{a_2})$ (with some $a_2 \gg1 $)
and $\{m_j \}_{j = 1}^J\subset \Z^2$
with $|m_j|\gg N_1$
such that 
\begin{align}
& \Big\|\big\|   \jb{\tau_2}^{-b} \jb{\tau_3}^{-b}
\| \wt  \Hb_{3b}^{N_1}(n,n_2,n_3,\tau,\tau_2,\tau_3) \|_{\l^2_{n_2n_3} \to \l^2_n}  \big\|_{L^2_{\tau,\tau_2,\tau_3}}
\Big\|_{L^p(\O)}
\notag \\
& \quad \les
N_1
 \sup_{|m| \les  N_1}
\Big\|\big\|  \jb{\tau_2}^{-b} \jb{\tau_3}^{-b}
\| \wt  \Hb_{3b}^{N_1, m}(n',n_2,n_3',\tau,\tau_2,\tau_3) \|_{\l^2_{n_2, n_3', n'}}  \big\|_{L^2_{\tau,\tau_2,\tau_3}}
\Big\|_{L^p(\O)}\notag \\
& \quad \quad 
+ 
N_1^\frac{a_2}{2}
\sup_{j  = 1, \dots, J}
\Big\|\big\|  \jb{\tau_2}^{-b} \jb{\tau_3}^{-b}
\|  \wt \Hb_{3b}^{N_1, m_j}(n',n_2,n_3',\tau,\tau_2,\tau_3) \|_{\l^2_{n_2, n_3', n'}}  \big\|_{L^2_{\tau,\tau_2,\tau_3}}
\Big\|_{L^p(\O)}\notag \\
& \quad \les
p ^\frac 12 T^\frac 12 N_1^{\frac{a_2}{2}+ 6}
\label{RB15}
\end{align}

\noi
for any $p \ge 2$, 
where the second step follows from 
(a slight modification of) \eqref{RB8}, 
Then,~\eqref{RB5} in this subcase follows from interpolating \eqref{RB14} and \eqref{RB15}. 

It remains to prove \eqref{RB14}.
Since $b > \frac 12$, we can proceed as in 
\eqref{RT13}
by applying 
Lemma~\ref{LEM:RT}\,(iii)
(instead of 
Lemma~\ref{LEM:RT}\,(i)
used in 
the proof of Lemma \ref{LEM:RT3a} presented at the end of this section).
This proves \eqref{RB14}
(with $\dl = s- \eps$).

\medskip

This concludes the proof of Proposition \ref{PROP:RT3}.
\end{proof}

We conclude this section by presenting a proof of 
 Lemma \ref{LEM:RT3a}.

\begin{proof}[Proof of Lemma \ref{LEM:RT3a}]

Given dyadic $N, N_1, N_2, N_3 \ge 1$, write $\Hb_3^{\bf N}$
in \eqref{RB6} (see also \eqref{RB4})
as 
\begin{align}
\Hb_3^{\bf N}(n,n_2,n_3,\tau,\tau_2,\tau_3)
= \sum_{m \in \Z} \Hb^{{\bf N}, (m)}_3 (n, n_2, n_3 )\ft \eta_{_T} (\tau + \tau_2 - \tau_3 - m),   
\label{RTB1}
\end{align}

\noi
where $\Hb^{{\bf N}, (m)}_3$ is defined by 
\begin{align}
\begin{split}
 \Hb^{{\bf N}, (m)}_3 (n, n_2, n_3)
& =  \frac{\jb{n}^{s}}{ \jb{n_2}^{s}\jb{n_3}^{s}} 
\ind_{\substack{|n| \sim N\\|n_2| \sim N_2\\|n_3| \sim N_3}}
\sum_{\substack{n + n_2 - n_3= n_1 \\ n_1 \neq n,n_2}} 
\frac{g_{n_1} }{\jb{n_1}}
\ind_{|n_1|\sim N_1}
\ind_{ \phi (\bar n) = m}\\
& =  \frac{\jb{n}^{s}}{ \jb{n_2}^{s}\jb{n_3}^s} 
\sum_{n_1 \in \Z^2} 
h^{{\bf N}, (m)}(n, n_1, n_2, n_3)
\frac{g_{n_1}}{\jb{n_1}}.
\end{split}
\label{RTB2}
\end{align}

\noi
Here,  $h^{{\bf N}, (m)}$ is the base tensor defined in \eqref{baseT}. 
Hence, from \eqref{RTB1}, 
$\ft \eta_{_T}(\tau) = T \ft \eta (T\tau)$, 
Lemma~\ref{LEM:RT}\,(i) with \eqref{RTB2}, and Lemma \ref{LEM:basetensor}, we have
\[
\begin{split} 
& \sup_{\tau,\tau_2, \tau_3\in \R} \Big\|  \| 
 \ind_{Q_0} (n) \ind_{Q_2} (n_2)
 \ind_{Q_3} (n_3)
\Hb_3^{\bf N}(n,n_2,n_3,\tau,\tau_2,\tau_3) 
 \|_{\l^2_{n_2n_3} \to \l^2_n} \Big\|_{L^p(\O)}  \\
 & \quad \le \sup_{\tau,\tau_2, \tau_3\in \R} \sum_{m \in \Z}
|\ft \eta_{_T} (\tau + \tau_2 - \tau_3- m)| \\
& \hphantom{XXXXX}
\times \Big\| \| 
 \ind_{Q_0} (n) \ind_{Q_2} (n_2)
 \ind_{Q_3} (n_3)
 \Hb_3^{{\bf N}, (m)}(n,n_2,n_3)
\|_{\l^2_{n_2n_3} \to \l^2_n} \Big\|_{L^p(\O)}  \\
&\quad \les p  N_1^\eps N^s N_1^{-1}  N_2^{-s} N_3^{-s}
\sup_{m \in \Z}
\max\Big(  \|  h^{{\bf N}, (m)}  \|_{n_1n_2n_3   \to n },
 \|  h^{{\bf N}, (m)}  \|_{n_2n_3   \to n n_1}\Big)
  \\
%& \les p 
%N_{\max}^\eps N^s N_1^{-1} N_3^{-1} N_2^{-s}   (N_1 N_3)^{\frac12 + \eps} \\
& \quad \les p 
 N_1^{-s+ 3\eps}, 
\end{split}
\]

\noi
where, in the last step,  we used the fact that $N_1 \ges N_{\med}$.
This proves  \eqref{RB10}.
\end{proof}

\section{Random tensor term IV}
\label{SEC:RT4}

In this section, we briefly discuss how to treat 
the  random tensor term:
\begin{align*}
\TT_4(w_1, w_3) =  \NN (w_1, z,w_3), 
%\label{RC0}
\end{align*}

\noi
where $z$ denotes the random linear solution in \eqref{lin1}.

\begin{proposition}
\label{PROP:RT4}

Given small   $s > 0$ and $\eps > 0$, 
 there exists $\ta > 0$
such that 
\begin{align*}
\Big\| \| \TT_4 \|_{X^{s,b}_T\times X^{s,b}_T  \to 
X^{s,-b'}_T}  \Big\|_{L^p (\O)} 
\les p^\frac 12   T^\ta
%\label{RC0a}
\end{align*}

\noi
for any finite $p \ge 1$ and $0 < T\ll 1$, where 
$b = \frac 12 + \eps$ and $b' = \frac 12 - 2\eps$
are as in \eqref{b1}.
In particular, 
 the conclusion of Proposition \ref{PROP:1}
holds
for the random tensor term $\NN(w, z,  w)$.

\end{proposition}

\begin{proof}

Given dyadic $N_2 \ge 1$, define
\begin{align*}
\TT_4^{N_2} (w_1, w_3) =  \NN(w_1, \Q_{N_2} z, w_3).
%\label{RC1}
\end{align*}

\noi
Then, it suffices to prove that 
there exist 
 small $\ta, \dl > 0$ such that 
\begin{align}
\Big\| \| \TT_4^{N_2} 
\|_{X^{s,b}_T\times X^{s,b}_T  \to X^{s,-b'}_T} 
\Big\|_{L^p (\O)} 
\les p^\frac 12   T^\ta N_2^{-\dl} 
\label{RC2}
\end{align}

\noi
for any finite $p \ge 1$ and $0 < T\ll 1$.

\medskip

\noi 
$\bullet$ {\bf Case 1:}
$N_2  \les \min(|n_1|, |n_3|)$ or 
$|\phi(\bar n)| \gg  N_2^2$.
\\
\indent
A slight modification of the argument in Case 1 
or Case 2 of the proof of Proposition \ref{PROP:RT3}
yields \eqref{RC2}.

\medskip

As seen in the proof of Proposition \ref{PROP:RT3}, 
 \eqref{RC2}
follows once we prove 
\begin{align}
\begin{split}
& \Big\|
\big\|  \jb{\tau}^{-b'} \jb{\tau_1}^{-b} \jb{\tau_3}^{-b}
\| \Hb_4^{N_2}(n,n_1,n_3,\tau,\tau_1,\tau_3) \|_{\l^2_{n_1, n_3} \to \l^2_n}  \big\|_{L^2_{\tau,\tau_1,\tau_3}}
\Big\|_{L^p(\O)}\\
& \quad \les p^\frac 12   T^\ta N_2^{-\dl},  
\end{split}
\label{RC5}
\end{align}

\noi 
where 
$\Hb_4^{N_2}$ is given by 
\begin{align*}
%\begin{split}
\Hb_4^{N_2} (n,n_1,n_3,\tau,\tau_1,\tau_3)  
& = 
\frac{\jb{n}^s}{\jb{n_1}^s \jb{n_3}^s} \\
& \quad \times  \sum_{\substack{n - n_1 - n_3 = n_2\\
n_2 \ne n_1, n_3}}
\frac{\cj{g_{n_2}}  }{\jb{n_2}}  \ind_{|n_2|\sim N_2}
 \ft \eta_{_T}  (\tau - \tau_1 - \tau_3 - \phi (\bar n)).
%\end{split}
%\label{RC4}
\end{align*}

\noi
Given dyadic $N, N_1, N_2, N_3 \ge 1$, 
we set 
\begin{align*}
%\begin{split}
\Hb_4^{\bf N}(n,n_1,n_3,\tau,\tau_1,\tau_3)
& = \ind_{|n|\sim N}
\bigg(\prod_{j \in\{1, 3\}}
\ind_{|n_j|\sim N_j}
\bigg)\Hb_4^{N_2}(n,n_1,n_3,\tau,\tau_1,\tau_3).
%\end{split}
%\label{RC6}
\end{align*}

\noi
Then, by writing 
\begin{align}
\Hb_4^{N_2}
= 
\sum_{\substack{N, N_1, N_3 \ge 1, \text{ dyadic}\\ N_2 \ges \max (N_1, N_3)}} 
\Hb_4^{\bf N}
+ 
\sum_{\substack{N, N_1, N_3 \ge 1, \text{ dyadic}\\ N_2 \ll \max (N_1, N_3)}} 
\Hb_4^{\bf N}, 
\label{RC7}
\end{align}

\noi
 we separately estimate
the first and second terms on the right-hand side of~\eqref{RC7}
(after removing the contribution from Case 1).

\medskip

\noi 
$\bullet$ {\bf Case 2:}
$N_2\ges \max(N_1, N_3)$.
\\
\indent
We can conclude \eqref{RC5}
by 
an interpolation argument 
as in 
 Case 3 of the proof of Proposition~\ref{PROP:RT3},
 once we have  following lemma.

\begin{lemma}\label{LEM:RT4a}
Given dyadic $N, N_1, N_2, N_3 \ge 1$, 
suppose that $N_2 \ges N_{\med}$, where $N_{\med}$ is as in~\eqref{ord1}.
Let $Q_0$, $Q_1$, and $Q_3$
be cubes of side length $\les N_2$.
Fix small $s > 0$.
Then, we have 
\begin{align*}
&  \sup_{\tau,\tau_1,\tau_3\in \R} 
 \Big\| \| 
\ind_{Q_0} (n) \ind_{Q_1} (n_1) \ind_{Q_3} (n_3) \\
&  \hphantom{XXXXX}
\times
\Hb_4^{\bf N}(n,n_1,n_3,\tau,\tau_1,\tau_3) 
 \|_{\l^2_{n_1, n_3} \to \l^2_n} \Big\|_{L^p(\O)}  \les 
p^\frac 12  N_2^{- s + \eps}
\end{align*}

\noi 
for any finite $p \ge 1$
and  dyadic $N, N_1,  N_2, N_3  \ge 1$, 
uniformly in $0 < T \ll 1$
and the cubes $Q_0$, $Q_1$, and $Q_3$.
\end{lemma}

Since Lemma \ref{LEM:RT4a}
follows
from a minor modification of the proof of Lemma \ref{LEM:RT3a}, 
we omit details.

\medskip

\noi 
$\bullet$ {\bf Case 3:}
$N_2  \ll \max(N_1, N_3)$.
\\
\indent
By assuming $N_1 \ge N_3$ without loss of generality, 
we have $N\sim N_1 \gg N_2 \gg N_3$.
In view of Case 1, we may assume that 
$|\phi(\bar n)| \les  N_2^2$.
Since $\big||n_2|^2 - |n_3|^2\big| \les N_2^2$, 
we obtain
$\big||n|^2 - |n_1|^2\big| \les N_2^2$
in this case, 
and thus we can proceed as in Subcase 4.b of the proof 
of Proposition \ref{PROP:RT3}, 
using an interpolation argument
and 
Lemma~\ref{LEM:RT}\,(iii)
(instead of 
Lemma~\ref{LEM:RT}\,(i)
which is needed for a proof of 
Lemma \ref{LEM:RT4a}).
\end{proof}

\appendix

\section{Auxiliary lemma}

In this appendix, 
we prove an auxiliary lemma used in the proof of Lemma \ref{LEM:DH}.

\begin{lemma}\label{LEM:Y1}

Given $r \in \N$, fix
 $y \in \R^r$.
Given $N \gg1$, let $\{\al_i \}_{i=1}^r \subset \R^r$ be a collection of linearly independent  vectors,
 satisfying 
$ |\al_i| \les  N$, $i = 1, \dots, r$, 
such that 
\begin{align}
\textup{Vol}_r
\big(\Dl(\{\al_i\}_{i = 1}^r)\big) \ge \frac1{r!}, 
\label{aux0}
\end{align}

\noi
where $\Dl(\{\al_i\}_{i = 1}^r)$ denotes the 
$r$-dimensional simplex formed by 
the vectors $\al_1, \dots, \al_r$, 
and $\textup{Vol}_r$ denotes the $r$-dimensional volume.
Suppose that
\begin{align}
|y \cdot \al_i| \les N^{a_1}, \quad i = 1, \dots, r
\label{aux1}
\end{align}

\noi
for some $a_1 \ge 1$.
Then, there exists $C_r \gg 1$, independent of $N\gg1 $,
$\{\al_i \}_{i=1}^r \subset \R^r$,  and $a_1\ge 1$,  such that $|y| \les N^{C_r a_1}$.
\end{lemma}

\begin{proof}
We argue by induction.
When $r = 1$, we have  $\al_1 \in \R\setminus\{0\}$.
Then, 
given $y \in\R$ satisfying \eqref{aux1}, 
the conclusion follows from $|y \al_1| \les N^{a_1}$ and $|\al_1| \ge 1$.

Now, suppose that the conclusion holds for some $r \ge 1$.
%We turn to the case $r +1$.
 Let $\{\al_i \}_{i=1}^{r+1} \subset \R^{r+1}$ be a collection of linearly independent  vectors, 
  satisfying 
$ |\al_i| \le c_0  N$, $i = 1, \dots, r+1$, 
for some $c_0 > 0$
such that 
\begin{align}
\textup{Vol}_{r+1}
\big(\Dl(\{\al_i\}_{i = 1}^{r+1})\big) \ge \frac 1{(r+1)!}, 
\label{aux0a}
\end{align}

\noi
Define the hyperplane $\mathcal{P}$ by 
\[
\mathcal P = {\rm span} \{\al_i \}_{i=1}^r \subset \R^{r+1}.
\]

\noi
Given  $y \in \R^{r+1}$ satisfying
\begin{align}
|y \cdot \al_i| \le c_1 N^{a_1}, \quad i = 1, \dots, r+1, 
\label{aux2}
\end{align}

\noi
for some $c_1 > 0$, 
write it as 
\begin{align*}
y = y^\parallel + y^\perp,
%\label{YY1}
\end{align*}

\noi
where $y^\parallel \in \mathcal P$ and $ y^\perp \perp \mathcal P$.
From \eqref{aux2}, we have
\begin{align}
|y^\parallel \cdot \al_i | = |y \cdot \al_i | \le c_1 N^{a_1}
\label{aux0b}
\end{align}

\noi
for $i = 1, \dots, r$.
Now, 
 recall that the volume of a $(r+1)$-dimensional simplex $A$ is given by 
\begin{align}\label{volume}
\textup{Vol}_{r+1}(A)
= 
 \frac1{r+1} \cdot (\textup{base volume}) \cdot (\textup{height}).
\end{align}

\noi
Then, 
from \eqref{volume}, \eqref{aux0a},  and $|\al_{r+1}|\le c_0 N$, we see that 
\begin{align*}
\textup{Vol}_r
\big(\Dl(\{\al_i\}_{i = 1}^r)\big) \ge 
\frac{r+1}{|\al_{r+1}|}
\textup{Vol}_{r+1}
\big(\Dl(\{\al_i\}_{i = 1}^{r+1})\big) \ge (c_0N)^{-1}\cdot \frac 1 {r!}.
\end{align*}

\noi
By setting $\be_i = c_0^\frac 1r N^{\frac 1r} \al_i$, $i = 1, \dots, r$,
we see that $\be_i$'s are linearly independent,  
$|\be_i| \le c_0^{1 + \frac 1r} N^{1 + \frac 1r} =: c_0M$
for any 
$i = 1, \dots, r$,   and 
$\textup{Vol}_r
\big(\Dl(\{\be_i\}_{i = 1}^r)\big) \ge \frac1{r!}$.
Moreover, from~\eqref{aux0b} with $a_1 \ge 1$, we have 
\begin{align*}
|y^\parallel \cdot \be_i| \le c_1 c_0^\frac 1 r  N^\frac 1r N^{a_1}
\le c_1 M^{a_1}
\end{align*}

\noi
for any $i = 1, \dots, r$.
Hence, by the inductive hypothesis, we conclude that 
\begin{align}
|y^\parallel| \les M^{C_r a_1} \sim N^{C_r(1+ \frac 1r) a_1}
\label{YY2}
\end{align} 

\noi
for some $C_r \gg 1$.

It remains  to show that $|y^\perp| \les N^{C_r' a_1}$ for some $C_r' > 0$.
We prove this by contradiction.
Suppose that for any $k \in \N$, 
there exists $N = N(k) \gg1$ and $y_k\in \R^{r+1}$, satisfying
\eqref{aux1},  such that
\begin{align}
|y_k^\perp| \gg N^{ka_1},
\label{YY3} 
\end{align}

\noi
 where $y_k^\perp$ is the component of $y_k$ orthogonal to the hyperplane $\mathcal P$.
%%Suppose that since, otherwise, it follows from \eqref{aux2} that 
%%$|y_k'| \cdot |\al_{r+1}|
%%= 
%%|y_k' \cdot \al_{r+1}|  = |y_k \cdot \al_{r+1}|  \les N^{a_1}$, which implies $|y_k'| \les N^{a_1}$.
%%
%%
%%
%%
%%In the following, we assume that $y_k'$ and $\al_{r+1}$ are not colinear, 

Recall that  $\al_{r+1}$ is not on the hyperplane $\mathcal P$.
Write $\al_{r+1} = \al_{r+1}^\parallel + \al_{r+1}^\perp$, 
where 
$\al_{r+1}^\parallel \in \mathcal P$ and $\al_{r+1}^\perp\perp \mathcal P$.
Then, 
 we claim that
\begin{align}\label{dist}
|\al_{r+1}^\perp| = {\rm dist} (\al_{r+1}, \mathcal P) \ges N^{-r}.
\end{align}

\noi
Since $ |\al_i| \les  N$, $i = 1, \dots, r$, 
we have 
\begin{align}
\textup{Vol}_r\big(\Dl(\{\al_i\}_{i = 1}^r)\big) \les N^r.
\label{YY4}
\end{align}

\noi
Then, \eqref{dist} follows from \eqref{volume}
and \eqref{YY4}.

From \eqref{YY3} and \eqref{dist}, we have 
\begin{align}
|y_k^\perp \cdot \al_{r+1}| = |y_k^\perp| |\al^\perp_{r+1}| 
 \ges N^{k  a_1} N^{-r} \gg N^{\frac 12 k a_1}
\label{YY10}
\end{align}

\noi
for any $N \gg1$, provided that $k > 2r$.
Hence, 
from \eqref{YY2} and  \eqref{YY10} with $|\al_{r+1}| \les N$, 
we see that there exist $c_2 \gg c_3 > 0$ such that 
\[
|y_k \cdot \al_{r+1}| \ge |y_k^\perp \cdot \al_{r+1}| - |y_k^\parallel \cdot \al_{r+1}| 
\ge c_2 N^{\frac 12 k a_1} - c_3 N^{C_r (1+ \frac 1r) a_1 + 1}
\gg N^{a_1} 
\]

\noi
for any $N \gg 1$, provided that 
$k > 2  C_r (1+ \frac 1r) +2$, 
which is a contradiction to \eqref{aux2}.
Therefore, 
there exists  $C_{r+1}' > 0$ such that, 
given any $y \in \R^{r+1}$
satisfying~\eqref{aux2}, 
we have 
 $|y^\perp| \les N^{C_{r+1}' a_1}$.
 Together with \eqref{YY2}, we conclude that  
there exists $C_{r+1} > 0$ such that 
 $|y|\les N^{C_{r+1}a_1}$ for 
 any $y \in \R^{r+1}$
satisfying~\eqref{aux2}.
We note that, from the discussion above, 
it suffices to choose $C_{r+1}' > \max(2r, 
2C_r (1+ \frac 1r) +2)$, 
which shows that the choice of $C'_{r+1}$ (and hence of $C_{r+1}$)
is independent of $N \gg1$, 
$\{\al_i \}_{i=1}^{r+1}$,  and $a_1\ge 1$.

Finally, by induction, we 
 concludes the proof of Lemma \ref{LEM:Y1}.
\end{proof}

\begin{remark}\label{REM:Y2}\rm

Let  $\{\al_i \}_{i=1}^r \subset \Z^r$ be a collection of linearly independent  {\it integer} vectors,
 satisfying 
$ |\al_i| \les  N$, $i = 1, \dots, r$.
Then, 
we have 
\begin{align*}
\textup{Vol}_{r}
\big(\Dl(\{\al_i\}_{i = 1}^{r})\big) 
= \frac1{r!} |{\rm det} (\al_1,\al_2,\cdots, \al_{r})|\ge \frac1{r!}, 
\end{align*}

\noi
and thus the hypothesis \eqref{aux0} is automatically satisfied.

\end{remark}

\section{Declarations}

\noi
{\bf Funding.}
T.O.~was supported by the European Research Council (grant no. 864138 ``SingStochDispDyn")
and
 by the EPSRC 
Mathematical Sciences
Small Grant  (grant no.~EP/Y033507/1).
Y.W.~was supported by 
 the EPSRC New Investigator Award 
 (grant no.~EP/V003178/1).

%T.O.~and~G.Z.~were supported by the European Research Council
%(grant no. 864138 ``SingStochDispDyn'').
%T.O.~was also supported 
% by the EPSRC 
%Mathematical Sciences
%Small Grant  (grant no.~EP/Y033507/1).
%Y.W. was supported by  the EPSRC New Investigator
%Award (grant no. EP/V003178/1).

\medskip

\noi
{\bf Competing interests.}
The authors have no competing interests to declare that are relevant to the content of this article.

\medskip

\noi
{\bf Data availability statement.}
This manuscript has no associated data.

%The above should be summarized in a statement and placed in a ?Declarations? section before the reference list under a heading of ?Funding? and/or ?Competing interests?. Other declarations include Ethics approval, Consent, Data, Material and/or Code availability and Authors? contribution statements.

\begin{ackno}\rm
T.O. would like to thank Prof.~Yoshio Tsutsumi 
for his continuous support over the last twenty years
since the first time they met at MSRI in 2005.
The authors would like to thank Shao Liu for his careful
reading of the paper.
% The authors would like to express their gratitude to the anonymous referees for the helpful comments which improved the quality of the paper

\end{ackno}

%\begin{ackno}\rm
%T.O. would like to thank Prof.~Yoshio Tsutsumi 
%for his continuous support over the last twenty years
%since the first time they met at MSRI in 2005.
%T.O.~was supported by the European Research Council (grant no. 864138 ``SingStochDispDyn")
%and
% by the EPSRC 
%Mathematical Sciences
%Small Grant  (grant no.~EP/Y033507/1).
%Y.W.~was supported by 
% the EPSRC New Investigator Award 
% (grant no.~EP/V003178/1).
%\end{ackno}
%
%

\end{document}